\documentclass[AIF,Unicode,manuscript]{cedram}

\def\nc{\newcommand}

\def\ep{\epsilon}

 \def\Om{\Omega}

\def\x{\times}
\nc\pa{\partial}

\nc\CC{\mathbb{C}}
\nc\RR{\mathbb{R}}
\nc\QQ{\mathbb{Q}}
\nc\ZZ{\mathbb{Z}}
\nc\NN{\mathbb{N}}

\nc\m[1]{\left| #1\right|}
\nc\norm[1]{\left\| #1\right\|}
\def\nc{\newcommand}

\def\ep{\epsilon}
\def\bff{{\bf f}}

 \def\Om{\Omega}

\def\integral{\int}%\limits}% Change this to \int to obtain integral as before
\def\aa{\mathcal{A}}

\def\pdl{p-\delta}
\def\x{\times}
\def\pmd{p-\delta}

\def\bea{\begin{equation}\begin{aligned}}
\def\ena{\end{aligned}\end{equation}}

\def\beas{\begin{equation*}\begin{aligned}}
\def\enas{\end{aligned}\end{equation*}}

\nc\axgrad[1]{\mathcal{A}(x, #1)}

\title[BMO solutions to quasilinear equations]
{BMO solutions to quasilinear equations\\ of $p$-Laplace type}

\author{\lastname{Nguyen} \middlename{Cong} \firstname{Phuc}}
\address{Department of Mathematics,
	Louisiana State University,
	303 Lockett Hall, Baton Rouge, LA 70803, USA.}
\email{pcnguyen@math.lsu.edu}
\thanks{N. C. Phuc is supported in part by Simons Foundation, award number 426071.}

\author{\firstname{Igor} \middlename{E.} \lastname{Verbitsky}}
\address{Department of Mathematics,
	University of Missouri,
	Columbia, MO 65211, USA.}
\email{verbitskyi@missouri.edu}



\begin{abstract} 
 We give necessary and sufficient conditions 
 for the existence of a BMO solution to the quasilinear equation   
 $-\Delta_{p} u  = \mu$   in $\RR^n$,  $u\ge 0$,  
 where $\mu$ is a locally finite Radon measure,  and  $\Delta_{p}u= \text{div}(|\nabla u|^{p-2}\nabla u)$ is the $p$-Laplacian ($p>1$). 
 
 We also characterize BMO solutions to  equations 
 $
 -\Delta_{p} u  = \sigma u^{q}  + \mu$  in $\RR^n$, $u\ge 0$,  with  $q>0$,
 where both $\mu$ and $\sigma$ are locally finite Radon measures.  Our main results hold for a class of more general quasilinear  operators 
 ${\rm div}(\mathcal{A}(x, \nabla \cdot))$ in place of  $\Delta_{p}$.
\end{abstract}

\begin{altabstract}
    Nous donnons les conditions n\'ecessaires et suffisantes pour l'existence d'une solution BMO de l'\'equation quasi-lin\'eaire
    $ - \Delta_ {p} u = \mu $ dans $\mathbb{R}^n$, $ u \ge 0 $, o\`u $\mu$ est une mesure de Radon localement finie, et
    $\Delta_ {p} u = \text {div} (| \nabla u |^ {p-2} \nabla u)$ est le $p$-Laplacien ($p> 1$).

    Nous caract\'erisons \'egalement les solutions BMO de l'\'equation 
    $- \Delta_ {p} u = \sigma u^ {q} + \mu $ in $ \mathbb{R}^n $, $ u \ge 0 $, avec $q> 0$,
    o\`u $\mu$ et $\sigma$ sont des mesures de Radon localement finies. Nos principaux r\'esultats sont valables pour une classe d'op\'erateurs quasi-lin\'eaires plus g\'en\'eraux
    ${\rm div} (\mathcal {A} (x, \nabla \cdot))$ \`a la place de $\Delta_ {p}$.
\end{altabstract}    
    
\begin{document}
\maketitle

\section{Introduction}%\label{Introduction}

Let  $M^{+}(\RR^n)$ denote the class of all (locally finite) positive  Radon measures in $\RR^n$,  $n\geq 2$. Let $\mu \in M^{+}(\RR^n)$.  In this paper, the following quasilinear equation with measure data is considered:
\begin{equation}\label{Basic-PDE}
\left\{ \begin{array}{ll}
-\Delta_p u = \mu, \quad u\geq 0  \quad \text{in } \RR^n, \\
\displaystyle{\liminf_{|x|\rightarrow \infty}}\,  u = 0. 
\end{array}
\right.
\end{equation}
Here $\Delta_p u$ is the $p$-Laplacian of $u$ defined by $\Delta_p u:= {\rm div} (|\nabla u|^{p-2} \nabla u$). All solutions to  \eqref{Basic-PDE} are understood to be $p$-superharmonic solutions, or equivalently local renormalized solutions (see \cite{HKM}, \cite{KKT}). Since nontrivial $p$-superharmonic 
functions on $\RR^n$ do not exist for $p\ge n$, it will be our standing assumption 
that solutions to \eqref{Basic-PDE} are considered for $1<p<n$.

It is known that  a necessary and sufficient condition for  \eqref{Basic-PDE} to 
admit a  solution is the finiteness condition (see, e.g., \cite{PV2})
\begin{equation}\label{finiteness}
\int_{1}^{\infty} \left(\frac{\mu(B(0,\rho))}{\rho^{n-p}} \right)^{\frac{1}{p-1}}\frac{d\rho}{\rho}<+\infty.
\end{equation} 
This is equivalent to the condition ${\bf W}_{p}\mu(x)<+\infty$ for some $x\in\RR^n$ (or equivalently
quasi-everywhere in $\RR^n$ with respect to the $p$-capacity),
where 
$${\bf W}_{p}\mu(x):=\int_0^\infty \left(\frac{\mu(B(x,\rho))}{\rho^{n-p}} \right)^{\frac{1}{p-1}}\frac{d\rho}{\rho}$$
is the Havin--Maz'ya--Wolff potential of $\mu$, often called the Wolff potential 
(see \cite{HeWo}, \cite{KuMi}, \cite{Maz}). 

By the important result of Kilpel\"{a}inen and Mal\'y \cite{KM2}, any solution $u$ to 
\eqref{Basic-PDE} satisfies the pointwise estimates 
\begin{equation}\label{K-M}
C_1 {\bf W}_{p}\mu(x) \le u(x) \le C_2 {\bf W}_{p}\mu(x), \quad x \in \RR^n, 
\end{equation} 
where $C_1, C_2 $ are positive constants that depend only on $p$ and $n$.

The $p$-capacity ${\rm cap}_p(\cdot)$ is a natural capacity associated with the $p$-Laplace operator  defined for each compact set $K$ of $\RR^n$ by
$${\rm cap}_p(K)=\inf\left\{\int_{\RR^n} |\nabla h|^p dx: \, \, h\in C_0^\infty(\RR^n), \,\, h\geq 1 \text{ on } K \right\}.$$

In \cite[Lemma 3.1]{V1}, it is shown that, for $1<p<n$, if  \eqref{Basic-PDE} has a solution $u\in {\rm BMO}(\RR^n)$ then \eqref{finiteness} holds along with the following bound for $\mu$:
\begin{equation}\label{mucond}
\mu(B(x,R))\leq C R^{n-p}, \quad \forall x\in\RR^n, \,  R>0.
\end{equation} 
Here ${\rm BMO}(\RR^n)$ is the space of functions $u$ of bounded mean oscillation in $\RR^n$, i.e., $u \in L^1_{\text{loc}}(\RR^n)$ such that 
$$
\frac{1}{|B|} \int_B |u -  \bar{u}_B| dx \le C, 
$$
for all balls $B$ in $\RR^n$, where $\bar{u}_B=\frac{1}{|B|} \int_B u \, dx$. 

It is also known  that, conversely,     if $\mu$ satisfies  \eqref{finiteness} and \eqref{mucond},  then \eqref{Basic-PDE} has a solution $u\in {\rm BMO}(\RR^n)$ provided $2-\frac{1}{n}<p<n$ (see \cite{V1}).  A local version of this result was  established in \cite{Min1} for $p> 2$. The linear case $p=2$ is due to 
D. Adams \cite{Ada}. 

One of the main goals of this paper is to extend this existence criterion to the full range 
$1<p<n$.

\begin{theorem}\label{thm1} Let  $\mu \in M^{+}(\RR^n)$ and $1<p<n$. Then equation \eqref{Basic-PDE} has a solution $u\in {\rm BMO}(\RR^n)$
	if and only if $\mu$ satisfies conditions \eqref{finiteness} and \eqref{mucond}. 
	
	Moreover, any solution $u$  to \eqref{Basic-PDE}  lies in  ${\rm BMO}(\RR^n)$
	if and only if $\mu$ satisfies  \eqref{mucond}.
\end{theorem}

\begin{remark}\label{rem-1} If 
	$\mu$ satisfies  \eqref{mucond}, then actually any solution       
	$u$ to \eqref{pq-PDE} satisfies  the Morrey condition 
	\begin{equation*} %\label{morrey} 
	\int_{B(x, R)} |\nabla u|^s dy \le C \, R^{n-s}, \quad  \forall x \in \RR^n, \,\, R>0, 
	\end{equation*}
	provided $0<s<p$, which yields $u\in {\rm BMO}(\RR^n)$ for any $s\ge 1$ 
	by Poincar\'{e}'s inequality.  
	
	A sharper local estimate for the end-point weak $L^p$ norm $||\nabla u||_{L^{p, \infty}(B(x, R))}$ in place of 
	Morrey's norm is obtained in Theorem \ref{Morreybound} below. 
	\end {remark}

	We observe that Theorem \ref{thm1} can be regarded as the end-point case $\alpha=0$ of the corresponding criterion for $u\in C^\alpha(\RR^n)$ ($0\le \alpha< 1$), where 
	$C^\alpha(\RR^n)$  is the Campanato space
	of functions $u \in L^1_{\text{loc}}(\RR^n)$ such that 
	$$
	\frac{1}{|B|} \int_B |u -  \bar{u}_B| dx \le C \, |B|^{\frac{\alpha}{n}}, 
	$$
	for all balls $B$ in $\RR^n$. 
	Then  $u\in C^\alpha(\RR^n)$  if and only if  $u \in {\rm BMO}(\RR^n)$ when $\alpha=0$, and  $u$ is $\alpha$-H\"older continuous when $\alpha\in(0,1]$ (see \cite[Sec. 2.3]{Giu}).

	A local version of the following result was obtained in \cite[Theorem 4.18]{KM2} and  
	\cite[Theorem 1.14]{KZh}: \smallskip

	\textit{Let $0<\alpha<1$. A solution $u$ to \eqref{Basic-PDE} is in $C^\alpha (\RR^n)$ 
		if and only if $\mu$ satisfies  the condition 
		\begin{equation}\label{mucond-a}
		\mu(B(x,R))\leq C R^{n-p + \alpha(p-1)}, \quad \forall x\in\RR^n, \, R>0.
		\end{equation} 
	}
	
	Notice that condition  \eqref{mucond-a} combined with  \eqref{finiteness}  is necessary and sufficient for the existence of a solution $u\in C^\alpha(\RR^n)$ 
	to \eqref{Basic-PDE}. The proof is similar to the  
	proof of  Theorem \ref{thm1} given below for $\alpha=0$.

	Using Theorem \ref{thm1}, we  obtain criteria for the existence of BMO solutions to the equation 
	\begin{equation}\label{pq-PDE}
	\left\{ \begin{array}{ll}
	- \Delta_p u = \sigma u^q + \mu, \quad u\geq 0  \quad \text{in } \RR^n, \\
	\displaystyle{\liminf_{|x|\rightarrow \infty}}\,  u = 0, 
	\end{array}
	\right.
	\end{equation}
	where $q>0$ and $\mu, \sigma\in M^{+}(\RR^n)$ ($\sigma\not=0$). Here we  assume 
	without loss of generality that $\sigma\not=0$, since the case $\sigma=0$ is covered 
	by Theorem \ref{thm1}. 
	
	The corresponding results 
	are more complicated due to 
	the possible interaction between the datum $\mu$ and 
	the source term $\sigma u^q$ on the right-hand side, as well as competition with 
	$-\Delta_p u$ on the left-hand side. 
	
	We first consider the super-natural growth case $q>p-1$. 
	
	\begin{theorem}\label{thm2} Let  $\mu, \sigma \in M^{+}(\RR^n)$ ($\sigma\not=0$). 
		Let $q>p-1$ 
		and $1<p<n$. Then equation \eqref{pq-PDE} has a solution $u\in {\rm BMO}(\RR^n)$
		if $\mu$ satisfies  \eqref{mucond}, and 
		\begin{align*}
		& (a) \quad {\bf W}_{p}[ ({\bf W}_{p} \mu)^{q}  d \sigma](x) \le c \,  {\bf W}_{p} \mu(x),  \quad \forall x\in\RR^n,  \\
		& (b) \quad \sigma (B(x, R))  \left[ \int_{R}^\infty \left( \frac{\mu(B(x, r))}{r^{n-p}}\right)^{\frac{1}{p-1}}
		\frac{dr}{r}\right]^q\le C \, R^{n-p},  \, \,  \forall x\in\RR^n, \, R>0, 
		\end{align*}
		where $c$, $C$ are positive constants, and $c=c(p, q, n)$ is sufficiently  small. 
		
		Conversely, if there exists a solution $u\in {\rm BMO}(\RR^n)$  to \eqref{pq-PDE}, 
		then  \eqref{mucond}, and conditions (a), (b)  
		hold for some positive constants $c$, $C$. 
	\end{theorem}
	
	In the sub-natural growth case $0<q<p-1$,  
	we denote by $\kappa$ the least constant in the weighted norm inequality for Wolff potentials (see \cite{CV}), 
	\begin{equation}\label{weighted-in}
	|| {\bf W}_{p} \nu||_{L^q(d \sigma)} \le \kappa \, ||\nu||^{\frac{1}{p-1}}, \quad \forall \nu\in M^{+}(\RR^n),
	\end{equation}
	where $ ||\nu||=\nu(\RR^n)$ stands for 
	the total variation of $\nu$.

	Using \eqref{K-M}, it is easy to see that $\kappa$ is equivalent to the least constant 
	$\varkappa$ in the  inequality 
	\begin{equation}\label{weighted-in-a}
	||\phi||_{L^q(d \sigma)} \le  \varkappa \, ||\Delta_p \phi||^{\frac{1}{p-1}},
	\end{equation}
	for all positive  test functions $\phi$ that are  $p$-superharmonic in $\RR^n$ 
	such that $\displaystyle{\liminf_{|x|\to \infty}} \, \phi(x)=0$. 
	
	For $\sigma \in M^{+}(\RR^n)$ and a ball $B$ in $\RR^n$, let $\sigma_B=\sigma|_B$ 
	be the restriction of $\sigma$ to  $B$. We denote by $\kappa(B)$ the least constant in the \textit{localized} version of \eqref{weighted-in}, namely, 
	\begin{equation}\label{bestcst}
	||{\bf W}_{p} \nu||_{L^q(d\sigma_{B}) }\le {\bf {\bf \kappa}}(B)\, ||\nu||^{\frac{1}{p-1}}, \quad \forall 
	\nu\in M_{+}(\RR^n).
	\end{equation}
	\smallskip
	As mentioned above,  equivalent constants $\varkappa({B})$, associated with $\sigma_B$  in place of $\sigma$ in \eqref{weighted-in-a},  
	can be used in place of $\kappa({B})$. Various lower and upper estimates of  $\kappa({B})$ can be found in \cite{CV}.

	\begin{theorem}\label{thm3} Let  $\mu, \sigma \in M^{+}(\RR^n)$ ($\sigma\not=0$). 
		Let $0<q<p-1$ and $1<p<n$. Then equation \eqref{pq-PDE} has a nontrivial solution 
		$u\in {\rm BMO}(\RR^n)$
		if and only if	$\mu$ satisfies condition \eqref{mucond}, and there exists 
		a constant $C$ such that 
		\begin{align*}
		& (a) \quad  \kappa(B(x, R))^{\frac{q(p-1)}{p-1-q}}\le C \, R^{n-p}, 
		\\
		& (b) \quad \sigma (B(x, R))  \left[ \int_{R}^\infty \left( \frac{\mu(B(x, r))}{r^{n-p}}\right)^{\frac{1}{p-1}}
		\frac{dr}{r}\right]^q\le C \, R^{n-p}, \\
		& (c) \quad  \sigma(B(x, R)) \left[ \int_R^\infty \left ( \frac{\sigma(B(x, r))}{r^{n-p}}\right)^{\frac{1}{p-1}} \frac{d r}{r}\right]^{\frac{q(p-1)}{p-1-q}} \le C \, R^{n-p},\\
		& (d) \quad \sigma(B(x, R))  \left [\int_R^\infty \left (\frac{\kappa(B(x, r))^{\frac{q(p-1)}{p-1-q}}}{r^{n-p}}\right)^{\frac{1}{p-1}} \frac{d r}{r} \right]^q \le C \, R^{n-p}, 
		\end{align*}
		for all $x\in \RR^n$ and $R>0$. 
		
		Moreover, under the above conditions on $\mu$ and $\sigma$ any solution $u$  to \eqref{pq-PDE} lies in ${\rm BMO}(\RR^n)$, and  satisfies the Morrey estimates 
		of Remark \ref{rem-1}.
	\end{theorem}
	
	We remark that Theorem \ref{thm3} (except for the last statement) was proved in \cite{V1} for $2-\frac{1}{n}<p<n$ in the special case $\mu=0$. Conditions on $\mu$ and $\sigma$ in Theorem \ref{thm3} can be simplified substantially  under the assumption 
	\begin{equation}\label{cap-K}
	\sigma(K) \le C \, {\rm cap}_p(K), \quad \forall \, \text{compact sets} \,  
	K \subset \RR^n.
	\end{equation}
	
	\begin{corollary}\label{cor-1} Let  $\mu, \sigma \in M^{+}(\RR^n)$, 
		where $\sigma\not=0$ satisfies condition \eqref{cap-K}.  
		Let $0<q<p-1$ and $1<p<n$. Then equation \eqref{pq-PDE} has a 
		nontrivial  solution $u\in {\rm BMO}(\RR^n)$
		if and only if	$\mu$ satisfies  \eqref{mucond}, and, for all $x \in \RR^n$, $R>0$, 
		\begin{align*}
		& (a) \quad   \sigma (B(x, R))  \left[ \int_{R}^\infty \left( \frac{\mu(B(x, r))}{r^{n-p}}\right)^{\frac{1}{p-1}}
		\frac{dr}{r}\right]^q\le C \, R^{n-p}, 
		\\ & (b) \quad 
		\sigma(B(x, R)) \left[ \int_R^\infty \left ( \frac{\sigma(B(x, r))}{r^{n-p}}\right)^{\frac{1}{p-1}} \frac{d r}{\rho}\right]^{\frac{q(p-1)}{p-1-q}} \le C \, R^{n-p}.
		\end{align*} 
		Moreover, under the above conditions on $\mu$ and $\sigma$ any solution $u$  to \eqref{pq-PDE}  lies in  ${\rm BMO}(\RR^n)$,  and  satisfies the Morrey estimates 
		of Remark \ref{rem-1}.
	\end{corollary}

	Similar criteria for the existence of ${\rm BMO}$ solutions 
	in the natural growth case $q=p-1$ are obtained in Sec. \ref{sec5} below (see Theorem \ref{thm4}) under some additional assumptions on $\sigma$ stronger than \eqref{cap-K}, which is necessary in that case.

	\begin{remark}
		Theorems  \ref{thm1}, \ref{thm2}, \ref{thm3} and Corollary \ref{cor-1} can  be extended to equations with more general quasilinear elliptic operators ${\rm div}\, (\mathcal{A}(x, \nabla \cdot))$ in place of  $\Delta_{p}$, as long as the nonlinearity $\mathcal{A}(x,\xi)$ satisfies conditions \eqref{monotone} and \eqref{ellipticity22} below with $1<p<n$. This remark also applies to Theorem \ref{thm4}. 
	\end{remark}

\section{Proof of Theorem \ref{thm1}}%\label{sec2}

%\begin{proofof}
%\begin{proof}[Proof of Theorem \ref{thm1}] 

As mentioned above, the ``only if'' 
part of Theorem \ref{thm1} is proved in \cite[Lemma 3.1]{V1}. 

To prove the ``if'' part, we  construct a solution $u\in {\rm BMO}(\RR^n)$ to \eqref{Basic-PDE} under conditions \eqref{finiteness} and \eqref{mucond} on $\mu$.	Our construction below is based on an a priori estimate of \cite{AP}. Alternatively, it is  possible to use the gradient estimates of \cite{NP1}, \cite{NP2} for the construction,  but that will not be implemented  in this paper.

We first observe that under \eqref{mucond}, 
\begin{equation}\label{weakmorrey}
\norm{{\bf I}_1 \mu}_{L^{\frac{p}{p-1},\infty}(B(x,R))} \leq C M R^\frac{(n-p)(p-1)}{p}, \quad \forall x\in\RR^n, \, R>0,
\end{equation}
where $C$ depends only on $p$ and $n$, and
$$M =\sup_{x\in \RR^n, R>0} \frac{\mu(B(x,R))}{R^{n-p}}.$$
A proof of \eqref{weakmorrey} can be found in \cite{Ada}. Here ${\bf I}_\alpha$, 
$\alpha\in(0, n)$,  is the Riesz potential of order  $\alpha$, defined for a measure $\mu\in M^+(\RR^n)$  by
\begin{equation}\label{riesz}
{\bf I}_\alpha \mu(x):= \int_{\RR^n} \frac{1}{|x-y|^{n-\alpha}} d\mu(y), \qquad x\in\RR^n.
\end{equation}
Also, the space $L^{q,\infty}(B(x,R))$, $q>0$, is the weak $L^q$ space over the ball $B(x,R)$ with
$$\norm{f}_{L^{q,\infty}(B(x,R))}:=\sup_{\lambda>0} \lambda |\{y\in B(x,R): |f(y)|>\lambda\}|^{\frac{1}{q}}.$$

Let $\mu_k$ ($k=1,2,\dots$) be a standard regularization of $\mu$  by the convolution
$$\mu_k(x)= \rho_k*\mu(x), \quad \rho_k(\cdot)= k^n \rho(k\cdot), $$
where $0\leq \rho\in C_0^\infty(B(0,1))$, $\int_{\RR^n} \rho dx =1$, and $\rho$ is radial. Then it follows from Fubini's Theorem and \eqref{mucond} that 
\begin{equation*}
\mu_k(B(x,R))\leq C R^{n-p}, \quad \forall x\in\RR^n , \, R>0,
\end{equation*} 
where $C$ is independent of $k$. Thus, we also have 
\begin{equation}\label{weakmorreyk}
\norm{{\bf I}_1 \mu_k}_{L^{\frac{p}{p-1},\infty}(B(x,R))} \leq C R^\frac{(n-p)(p-1)}{p}, \quad \forall x\in\RR^n, \,  R>0,
\end{equation}
for a constant $C$ independent of $k$.

Next, for each positive integer $N$ we let $\mu_{B(0,N)}$ be the restriction of  the measure $\mu$ to the open ball $B(0,N)$.
Consider  now the unique $p$-superharmonic solution $u_{N,k}\in W^{1,p}_0(B(0,N))\subset W^{1,p}_0(\RR^n)$ to the equation
\begin{equation}\label{Approx-PDE1}
-\Delta_p u_{N,k} = \rho_k* \mu_{B(0,N)} \quad \text{in } B(0,N).
\end{equation}

Note  that $\rho_k* \mu_{B(0,N)}\leq \mu_k$ and we can write $\rho_k* \mu_{B(0,N)}= -{\rm div}\,  \nabla v_{N,k}$ in the sense of distributions in $B(0,N)$, where 
$$v_{N,k}(x)=\int_{B(0,N)} G_N(x,y) [\rho_k* \mu_{B(0,N)}(y)] dy, \qquad x\in B(0,N),$$  
with $G_N(x,y)$ being the Green function associated with $-\Delta$ in $B(0,N)$. Moreover, we have  
\begin{equation}\label{vI}
|\nabla v_{N,k}|\leq C\, {\bf I}_1(\rho_k* \mu_{B(0,N)}) \leq C\, {\bf I}_1\mu_k.
\end{equation}

Now using \eqref{weakmorreyk}, \eqref{vI} and  applying \cite[Theorem 1.2]{AP} we find that 
\begin{equation}\label{MorreyB}
\norm{\nabla u_{N,k}}_{L^{p, \infty}(B(x,R))}\leq C R^{\frac{n-p}{p}}
\end{equation}
for all $x\in B(0,N)$ and $0<R\leq 2N$. The constant $C$ is independent of $x,R,N$, and $k$. Since  $|\nabla u_{N,k}|=0$ outside $B(0,N)$, it is obvious that \eqref{MorreyB}
also holds for all $x\in \RR^n$ and $R>0$.

When restricted to the ball $B(0, N+1)$, $\mu_{B(0,N)}$ is a nonnegative finite measure and thus we  can write 
$$\mu_{B(0,N)}= f- {\rm div} F + \mu_s$$ 
as distributions in $B(0, N+1)$ (see, e.g., \cite{DMOP}). Here  $f\in L^1(B(0, N+1))$, $F\in L^{\frac{p}{p-1}}(B(0, N+1),\RR^n)$, and $\mu_s$ is a nonnegative measure 
concentrated on a set  of zero  $p$-capacity in $B(0,N)$.

For any $\varphi\in C_0^\infty(B(0,N))$ and $k\geq1$ we have $\rho_k*\varphi\in C_0^\infty(B(0,N+1))$. Thus it follows that 
\begin{align*}
&\int_{B(0, N)} \rho_k*\mu_{B(0,N)} \varphi dx = \int_{B(0, N+1)} \rho_k*\varphi d\mu_{B(0,N)} \\
&= \int_{B(0, N+1)} f \rho_k*\varphi  dx + \int_{B(0, N+1)} F \cdot \nabla(\rho_k*\varphi) dx + 
\int_{B(0, N+1)}  \rho_k*\varphi  d\mu_s\\
&= \int_{B(0, N)} \rho_k* f \varphi  dx + \int_{B(0, N)} \rho_k* F \cdot \nabla \varphi dx +  \int_{B(0, N)}  \rho_k*\mu_s \varphi dx.
\end{align*}

That is, 
$$\rho_k*\mu_{B(0,N)}= \rho_k*f- {\rm div} (\rho_k*F) + \rho_k*\mu_s$$ 
pointwise everywhere and as distributions in $B(0,N)$. 

As $\mu_s(\RR^n\setminus B(0,N))=0$, by the Lebesgue Dominated Convergence Theorem, 
we see that $\rho_k*\mu_s \rightarrow \mu_s$ in the narrow topology of measures in $B(0,N)$ (see \cite[Definition 2.2]{DMOP}). Moreover,   this and the above equality yield
$$\int_{B(0,N)}|{\rm div} (\rho_k*F)| dx \leq M,$$
where $M$ is independent of $k$.

At this point, in view of \eqref{Approx-PDE1}, we  apply \cite[Theorem 3.2]{DMOP}  (see also \cite[Remark 6.6]{PV1}) to find a subsequence $\{u_{N,k_j}\}$ of $\{u_{N,k}\}$ and a $p$-superharmonic function $u_N$ in $B(0,N)$ such that 
$u_{N,k_j}\rightarrow u_N$ a.e., $\nabla u_{N,k_j} \rightarrow \nabla u_N$ a.e. as $j\rightarrow\infty$, and $u_N$ solves 
the equation
\begin{equation*}
\left\{ \begin{array}{ll}
-\Delta_p u_N &= \mu_{B(0,N)}  \quad \text{in } B(0,N), \\
\quad \quad u_N &= 0 \quad \text{on } \partial B(0,N),
\end{array}
\right.
\end{equation*}
in the renormalized sense (see \cite{DMOP} for the notion of renormalized solutions).

Moreover, by \eqref{MorreyB} we have $u_N\in W^{1,s}_0(B(0,N))\subset W^{1,s}_0(\RR^n)$ for all $1\leq s<p$ and 
Fatou's Lemma yields that 
\begin{equation}\label{MorreyB2}
\norm{\nabla u_{N}}_{L^{p, \infty}(B(x,R))}\leq C R^{\frac{n-p}{p}}, \quad \forall \, 
x\in \RR^n, \, R>0.
\end{equation}

By \cite[Theorem 2.1]{PV1}  we  have  
\begin{equation}\label{pww}
u_{N}(x)\leq C\, {\bf W}_{p}\mu(x), \qquad \forall x\in\RR^n.
\end{equation}

By \cite[Theorem 1.17]{KM1} we can find a subsequence $\{u_{N_j}\}$ of $\{u_{N}\}$ and
a $p$-superharmonic function $u$ in $\RR^n$ such that $u_{N_j}\rightarrow u$ a.e. and $\nabla u_{N_j} \rightarrow \nabla u$ a.e. as $j\rightarrow \infty$.
Note that by  condition \eqref{finiteness} and \eqref{pww}, $u$ must be finite a.e. (or q.e.) and 
\begin{equation*}
u(x)\leq C\, {\bf W}_{p}\mu(x), \qquad \forall x\in\RR^n.
\end{equation*}

The weak continuity result of \cite{TW} yields that $u$ is a $p$-superharmonic solution of \eqref{Basic-PDE}. 
That $\liminf_{|x|\rightarrow \infty} u=0$ follows from the fact that $\inf_{\RR^n} u=0$ and the latter is a direct consequence of the pointwise bound (see \cite{KM1}, \cite{KM2}):
$$u(x)\leq C\, {\bf W}_{p}\mu(x)\leq C [u(x)-\inf_{\RR^n} u], \qquad \forall x\in\RR^n.$$

Applying \eqref{MorreyB2} with $\nabla u_{N_j}$ in place of $\nabla u_N$, and 
using    Fatou's Lemma  as $N_j \to \infty$,   we obtain 
\begin{equation}\label{weak-morrey}
\norm{\nabla u}_{L^{p, \infty}(B(x,R))}\leq C R^{\frac{n-p}{p}}, \quad \forall \, 
x\in \RR^n, \, R>0.
\end{equation}

It is worth mentioning here that for $p \geq 2$ estimate \eqref{weak-morrey} can also be inferred from the work \cite{Min2}.   
Thus, for any ball $B=B(x,R)$, by Poincar\'e's inequality and H\"older's inequality  we find
\begin{align*}
\frac{1}{|B|}\int_{B} \left|u(x)- \frac{1}{|B|}\int_{B} u(y)dy  \right| dx &\leq C R  \frac{1}{|B|}\int_{B} |\nabla u| dx\\
&\leq C R |B|^{-\frac{1}{p}}  \norm{\nabla u}_{L^{p, \infty}(B)}\\
&\leq C.
\end{align*}

This shows that $u\in {\rm BMO}(\RR^n)$ as desired.

%%%%%%%%%%%%%%%%%%%%%%%%%%%%%%%%%%%%%%%%%%%%%%%%%%%%%%%%%%%%%%%%%%%%%%%%%%%%%%%%%%%%%%%%%%%%%%%%%%%%%%%%%%%%%%%%%%%
It remains to show that any solution $u$ to \eqref{Basic-PDE} lies in ${\rm BMO}(\RR^n)$ provided 
$\mu$ satisfies condition \eqref{mucond}. We will actually prove the 
stronger estimate \eqref{weak-morrey} (see Remark \ref{rem-1}).

\begin{lemma}\label{decaylem}
	Let  $u$ be a nonnegative $p$-superharmonic solution of $-\Delta_p u=\mu$ with $\mu$ satisfying condition \eqref{mucond}. Then $|\nabla u|\in L^{q}_{\rm loc}(\RR^n)$ provided $0<q<p$. Moreover,
	for  any $0<q<p$, $0<\epsilon<p-1$ and  any ball $B(x,R)$ we have 
	\begin{equation}\label{grad}
	\Big(R^{q-n}\int_{B(x,R)}|\nabla u|^{q} dy\Big)^{\frac{1}{q}}\leq C\, \left( [\inf_{B(x, 2R)} u]^{\frac{p-1-\epsilon}{p}} + \inf_{B(x, 2R)} u \right),
	\end{equation}
	where the constant $C$ depends on $p, q, \epsilon, n$ and the constant in condition \eqref{mucond}.
	In particular,
	\begin{equation}\label{decayMor}
	\lim_{R\rightarrow +\infty} R^{q-n}\int_{B(0,R)}|\nabla u|^{q} dy =0.
	\end{equation}
\end{lemma}

\begin{proof}
	Let $u_k=\min\{u, k\}$, $k=1,2, \dots$ Then $u_k\in W^{1,p}_{\rm loc}(\RR^n)$ is a supersolution in $\RR^n$ and hence the weak Harnack 
	inequality  \cite{Tru} implies that
	$$\Big(\frac{1}{|B(x,R)|}\int_{B(x,R)}u_{k}^{s}dy \Big)^{\frac{1}{s}}\leq C\, \inf_{B(x,R)}u_k \leq C\, \inf_{B(x,R)}u$$
	for $0<s<\frac{n(p-1)}{n-p}$. Thus letting $k\rightarrow\infty$ we obtain
	\begin{equation}\label{whn}
	\Big(\frac{1}{|B(x,R)|}\int_{B(x,R)}u^{s}dy \Big)^{\frac{1}{s}}\leq C\, \inf_{B(x,R)}u.
	\end{equation}
	
	To continue, we recall the following result for $p$-supersolutions from \cite[Lemma 3.57]{HKM}:
	
	\textit{ For any nonnegative $p$-supersolution $v$ in an open set $\Om\subset\RR^n$, any $\epsilon>0$, and any 
		function $\varphi\in C_0^\infty(\Omega)$, it holds that 	
		\begin{equation}\label{357}
		\int_{\Om} |\nabla v|^p v^{-1-\epsilon} |\varphi|^p dx \leq (p/\epsilon)^p \int_{\Om} v^{p-1-\epsilon} |\nabla \varphi|^p dx.
		\end{equation}
	}
	Now let $0<q<p$ and fix an $\epsilon$ such that  $0<\epsilon<p-1$. Applying H\"{o}lder's inequality, and then using \eqref{357} with $v=u_k$ and an appropriate cut-off function $\varphi$ 
	supported in $B(x, 2R)$ 
	such that $\varphi=1$ on $B(x, R)$ and $|\nabla \varphi |\le C R^{-1}$, 
	we estimate
	\begin{eqnarray*}
		&&\int_{B(x,R)}|\nabla u_{k}|^q dy=\int_{B(x,R)}|\nabla u_{k}|^q u_{k}^{-(1+\epsilon)
			q/p}u_{k}^{(1+\epsilon)q/p}dy\\
		&\leq& \Big(\int_{B(x,R)}|\nabla u_{k}|^{p}u_{k}^{-1-\epsilon} dy \Big)^{q/p}
		\Big(\int_{B(x,R)}u_{k}^{(1+\epsilon)q/(p-q)}dy\Big)^{(p-q)/p}\\
		&\leq& C R^{-q} \Big(\int_{B(x,2R)}u^{p-1-\epsilon} dy \Big)^{q/p}
		\Big(\int_{B(x,R)}u^{(1+\epsilon)q/(p-q)}dy\Big)^{(p-q)/p}.
	\end{eqnarray*}
	
	Thus it follows from \eqref{whn} that 
	\begin{align}\label{nauq}
	\int_{B(x,R)}|\nabla u_{k}|^q dy&\leq  C R^{-q} \Big(R^n [\inf_{B(x,2R)} u]^{p-1-\epsilon} \Big)^{q/p} \times\\
	&\qquad \qquad \times \Big(\int_{B(x,R)}u^{(1+\epsilon)q/(p-q)}dy\Big)^{(p-q)/p}.\nonumber
	\end{align}
	
	On the other hand, by \cite[Theorem 1.6]{KM2} we have 
	\begin{align*}
	u(y) &\leq C \inf_{B(y, 3R)} u + C  \int_{0}^{6R} \left(\frac{\mu(B(y, t))}{t^{n-p}}\right)^{\frac{1}{p-1}} \frac{dt}{t}\\
	&\leq C \inf_{B(x, 2R)} u + C  \int_{0}^{6R} \left(\frac{\mu(B(y, t))}{t^{n-p}}\right)^{\frac{1}{p-1}} \frac{dt}{t},
	\end{align*}
	provided $y\in B(x,R)$. Thus,
	
	\begin{align}\label{usth}
	\int_{B(x,R)} & u^{(1+\epsilon)q/(p-q)}dy\leq C R^n [\inf_{B(x,2R)} u]^{(1+\epsilon)q/(p-q)} \\
	& + \, C \int_{B(x,R)} \left[\int_{0}^{6R} \left(\frac{\mu(B(y, t))}{t^{n-p}}\right)^{\frac{1}{p-1}} \frac{dt}{t} \right]^{(1+\epsilon)q/(p-q)}dy. \nonumber
	\end{align}
	
	Note that for $y\in B(x,R)$   by a Hedberg type inequality  (see \cite[Section 3.1]{AH}) we have 
	%\begin{align*}
	%\int_{0}^R \left(\frac{\mu(B(y, t))}{t^{n-p}}\right)^{\frac{1}{p-1}} \frac{dt}{t} &= \int_{0}^R \left(\frac{\mu_{B(x,2R)}(B(y, t))}{t^{n-p}}\right)^{\frac{1}{p-1}} \frac{dt}{t} \\
	%&\leq R^{\frac{p-\alpha}{p-1}} \int_{0}^R \left(\frac{\mu_{B(x,2R)}(B(y, t))}{t^{n-\alpha}}\right)^{\frac{1}{p-1}} \frac{dt}{t},
	%\end{align*}
	% provided $0<\alpha<p$. 
	% 
	
	\begin{align*}
	\int_{0}^{6R} & \left(\frac{\mu(B(y, t))}{t^{n-p}}\right)^{\frac{1}{p-1}} \frac{dt}{t} = \int_{0}^{6R} \left(\frac{\mu_{B(x,7R)}(B(y, t))}{t^{n-p}}\right)^{\frac{1}{p-1}} \frac{dt}{t} \\
	&\leq C \mu(B(x,7R))^{\frac{p-\alpha}{(n-\alpha)(p-1)}} \, {\bf M}_\alpha(\mu_{B(x,7R)})^{\frac{n-p}{(n-\alpha)(p-1)}}\\
	&\leq C \mu(B(x,7R))^{\frac{p-\alpha}{(n-\alpha)(p-1)}} \, {\bf I}_\alpha(\mu_{B(x,7R)})^{\frac{n-p}{(n-\alpha)(p-1)}},
	\end{align*}
	provided $0<\alpha<p$. Here ${\bf M}_\alpha$, $\alpha\in(0,n)$, is the fractional maximal function of order $\alpha$ defined for a  measure 
	$\nu\in M^{+}(\RR^n)$ by 
	$${\bf M}_\alpha \nu (x):=\sup_{r>0} \frac{\nu(B(x,r))}{r^{n-\alpha}}, \qquad x\in\RR^n.$$

	We now set $\theta=(1+\epsilon)q/(p-q)$. Then the above bound and \eqref{mucond} gives
	\begin{align*}
	\int_{B(x,R)} &\left[\int_{0}^{6R} \left(\frac{\mu(B(y, t))}{t^{n-p}}\right)^{\frac{1}{p-1}} \frac{dt}{t} \right]^\theta dy \\
	&\leq C  R^{\frac{(n-p)(p-\alpha)\theta}{(n-\alpha)(p-1)}} \int_{B(x,R)} {\bf I}_\alpha(\mu)^{\frac{\theta (n-p)}{(n-\alpha)(p-1)}} dy.
	\end{align*}

	% We now set $\theta=(1+\epsilon)q/(p-q)$ and without loss of generality we may assume that  $\theta>p-1$.
	% Then the above bound gives 
	%\begin{align*}
	%I&\leq C  R^{(\frac{p-\alpha)\theta}{p-1}} \int_{B(x,R)} \left[\int_{0}^R \left(\frac{\mu_{B(x,2R)}(B(y, t))}{t^{n-\alpha}}\right)^{\frac{1}{p-1}} \frac{dt}{t} \right]^{\theta} dy\\
	%&\leq C R^{(\frac{p-\alpha)\theta}{p-1}} \int_{B(x,4R)} [{\bf I}_{\alpha}\mu_{B(x,2R)}]^{\frac{\theta}{p-1}} dy. 
	%\end{align*}
	
	We next choose $0<\alpha<p$ such that    
	$$\frac{p}{p-\alpha}>\frac{\theta(n-p)}{(n-\alpha)(p-1)},$$
	and apply H\"older's inequality to get
	\begin{align}\label{I}
	\int_{B(x,R)} &\left[\int_{0}^{6R} \left(\frac{\mu(B(y, t))}{t^{n-p}}\right)^{\frac{1}{p-1}} \frac{dt}{t} \right]^\theta dy \\
	&\leq  C R^{\frac{(n-p)(p-\alpha)\theta}{(n-\alpha)(p-1)}} R^{n-\frac{n \theta(n-p)(p-\alpha)}{p(n-\alpha)(p-1)}} \norm{{\bf I}_{\alpha}\mu}^{\frac{\theta(n-p)}{(p-1)(n-\alpha)}}_{L^{\frac{p}{p-\alpha},\infty}(B(x,R))}\nonumber \\
	&  \leq C R^n, \nonumber
	\end{align}
	where in the last bound we used a result of \cite{Ada}: 
	$$\norm{{\bf I}_{\alpha}\mu}_{L^{\frac{p}{p-\alpha},\infty}(B(x,R))}\leq C R^{(n-p)\frac{p-\alpha}{p}}. $$
	
	At this point, we  plug  estimate \eqref{I}  into \eqref{usth} to obtain
	$$\int_{B(x,R)}  u^{(1+\epsilon)q/(p-q)}dy\leq C R^n(1+ [\inf_{B(x,2R)} u]^{(1+\epsilon)q/(p-q)} ).$$
	
	In view of \eqref{nauq}, this yields 
	$$R^{q-n}\int_{B(x,R)} |\nabla u_k|^q dy \leq C [\inf_{B(x,2R)} u]^{(p-1-\epsilon)q/p}\left(1+ [\inf_{B(x,2R)} u]^{(1+\epsilon)q/p}\right).$$

	Now  letting $k\rightarrow\infty$ we obtain estimate \eqref{grad}.
	
	Finally, to obtain the decay  \eqref{decayMor}, we observe that estimate \eqref{grad} also holds if $u$ is replaced by $\tilde{u}:=u-\inf_{\RR^n} u$ and that 
	$\inf_{\RR^n}\tilde{u}=0$.
\end{proof}

\begin{remark}
	For $0<q<\frac{n(p-1)}{n-1}$, Lemma \ref{decaylem}  holds without assuming condition \eqref{mucond} on $\mu$   (see \cite[Theorem 7.46]{HKM}). Moreover, 
	the first term on the right-hand side of 	\eqref{grad} can be dropped in this case.
\end{remark}

The next lemma is a local interior version of an analogous  result obtained in 
\cite[Proposition 4.4]{AP}. We use a modification of its proof  based mainly on  \cite[Theorem 2.3]{AP} and \cite[Lemma 2.8]{AP}. Henceforth, we denote by ${\bf M}$  the Hardy-Littlewood maximal operator.

\begin{lemma}\label{first-approx-lorentz}
	There exist constants $A=A(n,p)>1$ sufficiently large and  $\delta_0=\delta_0(n,p)\in (0,p-1)$ sufficiently small such that the following 
	holds for any $T>1$, $\lambda > 0$, and $\delta\in (0,\delta_0)$. Fix a ball $B_0=B(z_0, R_0)$ and let  $u$ be a  solution of $-\Delta_p u={\rm div}\, (|{\bf f}|^{p-2} {\bf f})$ in $2B_0$.  Assume that for some
	ball $B (y,\rho)$ with $\rho \leq R_0/8$, we have
	\beas
	B(y, \rho) \cap B_0 \cap \{x \in \RR^n : {\bf M} (\chi_{2B_0} &|\nabla u|^{p-\delta})^{\frac{1}{p-\delta}}(x) \leq \lambda \} \cap \\
	&\{{\bf M}(\chi_{2B_0} |\bff|^{\pdl})^{\frac{1}{p-\delta}} \leq \ep(T)\lambda \} \neq \emptyset, 
	\enas
	with $\ep(T) = T^{\frac{-2\delta_0}{\pdl}\,\max\left\{1, \frac{1}{p-1}\right\}}$.
	Then 
	\begin{equation*}%\label{largedecay}
	|\{ x \in \RR^n : {\bf M} (\chi_{2B_0}|\nabla u|^{\pdl})^{\frac{1}{\pmd}}(x) > A T \lambda \} \cap B(y,\rho)|  <  H \, |B(y,\rho)| ,
	\end{equation*}
	where
	\beas
	H =H(T,\delta)=  T^{-(p+\delta_0)} + \delta^{(p-\delta)\min\left\{1, \frac{1}{p-1}\right\}}. 
	\enas 
\end{lemma}

With this, we can now apply \cite[Lemma 4.1]{AP} and Lemma  \ref{first-approx-lorentz} above to get the following result. Its proof  is similar to that 
of \cite[Lemma 4.6]{AP}.

\begin{lemma}\label{technicallemma1} Let $A=A(n,p)$ and $\delta_0=\delta_0(n,p)$ be as in Lemma \ref{first-approx-lorentz}.
	The following holds for any $T > 2$ and $\delta\in (0, \delta_0)$. 
	Fix a ball $B_0=B(z_0, R_0)$ and let  $u$ be a  solution of $-\Delta_p u={\rm div}\, (|{\bf f}|^{p-2} {\bf f})$ in $2B_0$. Suppose that there
	exists $N>0$ such that 
	\begin{equation*}%\label{hypo1bdry}
	|\{x\in \RR^n: {\bf M} (\chi_{2B_0}|\nabla u|^{\pmd})^{\frac{1}{p-\delta}}(x) > N \}| < H | \tfrac{1}{8} B_{0}|.
	\end{equation*}
	Then for any  integer $k\geq 0$ we have 
	\beas
	&|\{x\in B_0: {\bf M}(\chi_{2B_0}|\nabla u|^{\pmd})^{\frac{1}{p-\delta}}(x)> N(AT) ^{k+1}\}|\\
	&\qquad\leq c(n)\, H\,  |\{x\in B_0: {\bf M}(\chi_{2B_0}|\nabla u|^{\pmd})^{\frac{1}{p-\delta}}(x) > N(AT) ^{k}\}|\\
	& \qquad \qquad+\, c(n)\, |\{ x\in B_0 : {\bf M}(\chi_{2B_0}|\bff|^{\pmd})^{\frac{1}{p-\delta}}(x) > \epsilon(T) N(AT)^{k}\}|.
	\enas
	Here $\epsilon(T)$ and $H=H(T,\delta)$ are as defined in Lemma \ref{first-approx-lorentz}.
\end{lemma}

We are now in a position to obtain a local $L^{p, \infty}$ estimate for the gradient.

\begin{theorem}\label{Lorentz-p-thick}
	Let $\mu\in M^{+}(\RR^n)$  and  let $u$ be a solution of $-\Delta_p u=\mu$ in $\RR^n$. Then  for any ball $B_0=B(z_0, R_0)\subset\RR^n$ we have
	\begin{eqnarray}\label{B0bound}
	\|\nabla u\|_{L^{p, \infty}(B_0))}&\leq& C\, |B_0|^{\frac{1}{p}}  \left(\frac{1}{|2B_0|}\int_{2B_0} |\nabla u|^{p-\delta} dx\right)^{\frac{1}{p-\delta}}   \\
	&& +\,  C\, \norm{[{\rm\bf I}_1(\chi_{2B_0}\mu)]^{\frac{1}{p-1}}}_{L^{p,\infty}(B_0)},\nonumber
	\end{eqnarray}
	with  a constant    $C=C(n,p)>0$ and a constant  $\delta=\delta(n,p)\in (0, p-1)$. 
\end{theorem}
\begin{proof}
	Let $B_0=B(z_0, R_0)$ and   $\delta_0$ be as in Lemma \ref{first-approx-lorentz}.
	For $T>2$ and  $\delta\in(0,\delta_0)$ to be determined,  we claim that there exists $N>0$  such that 
	\begin{equation*}
	|\{ x\in \RR^n: {\rm\bf M}(\chi_{2B_0}|\nabla u|^{p-\delta})^{\frac{1}{2-\delta}}(x) > N 
	\}| < H \, |\tfrac{1}{8}B_0|,
	\end{equation*}
	where $H =H(T,\delta)=  T^{-(p-\delta_0)} + \delta^{(p-\delta)\min\left\{1, \frac{1}{p-1}\right\}}$ (as in Lemma \ref{first-approx-lorentz}).
	
	This can be done by using the weak type $(1, 1)$ estimate for the maximal function and choosing $N>0$ such that 
	\begin{equation}\label{Neqn1}
	\frac{C(n)}{N^{p-\delta}}\int_{2 B_0}|\nabla u|^{p-\delta}dx =   H | \tfrac{1}{8} B_0|
	\end{equation}
	provided the integral above is non-zero, which we may assume.
	
	Let  $A>1$ and $\epsilon(T)>0$ be as in Lemma \ref{first-approx-lorentz}. Set 
	\begin{equation*}
	L = \sup_{k \geq 1} (AT)^{k} |\{ x\in B_0: {\rm\bf M}(\chi_{2B_0}|\nabla u|^{p-\delta})^{\frac{1}{p-\delta}}(x) > N(AT)^{k} \}|^{\frac{1}{p}}.
	\end{equation*}
	
	We have 
	\begin{equation}\label{SSS1}
	\norm{{\rm\bf M}(\chi_{2B_0}|\nabla u/N|^{p-\delta})^{\frac{1}{p-\delta}}}_{L^{p, \infty}(B_0)} \leq AT\, (|B_0|^{\frac{1}{p}} + L).
	\end{equation}

	We now set, for $m=1,2, \dots$, 
	\begin{equation*}
	L_m = \sup_{1\leq k \leq m} (AT)^{k} |\{ x\in B_0: {\rm\bf M}(\chi_{2B_0}|\nabla u|^{p-\delta})^{\frac{1}{p-\delta}}(x) > N(AT)^{k} \}|^{\frac{1}{p}},
	\end{equation*}
	and note that 
	\begin{equation}\label{limL}
	\lim_{m\rightarrow\infty}L_m =L.
	\end{equation}
	
	For any vector field ${\bf f}$ such that ${\rm div}\, (|{\bf f}|^{p-2} {\bf f})=\mu$ in $2B_0$, by  Lemma \ref{technicallemma1} we find 
	\begin{eqnarray*}
		L_m &\leq & C \sup_{1\leq k \leq m} \,  (AT)^{k} H(T,\delta)^{\frac{1}{p}} \times \\
		&& \qquad \qquad \times |\{ x\in B_0: {\rm\bf M}(\chi_{2B_0}|\nabla u|^{p-\delta})^{\frac{1}{p-\delta}}(x) > N(AT)^{k-1} \}|^{\frac{1}{p}}\\
		&& +\, C \sup_{1\leq k \leq m}\, (AT)^{ k}  \times \\ 
		&& \qquad \qquad \times  |\{ x\in B_0 : [{\rm\bf M}(\chi_{2B_0} |{\bf f}|^{p-\delta})]^{\frac{1}{p-\delta}} > \epsilon(T) N(AT)^{k-1}\}|^{\frac{1}{p}}\\
		&\leq& C\, (AT) H(T,\delta)^{\frac{1}{p}} (L_m+ |B_0|^{\frac{1}{p}}) \\
		&& \qquad \qquad   +\,  C_1(T,\delta)\,  \norm{[{\rm\bf M}(\chi_{2B_0} (|{\bf f}|/N)^{p-\delta})]^{\frac{1}{p-\delta}}}_{L^{p,\infty}(B_0)}.
	\end{eqnarray*}

	By the boundedness property of ${\bf M}$, this yields
	\begin{equation*}
	L_m 		\leq C\, (AT) H(T,\delta)^{\frac{1}{p}} (L_m+ |B_0|^{\frac{1}{p}})  + C C_1(T, \delta) \norm{ {\bf f}/N}_{L^{p, \infty}(2B_0)}.
	\end{equation*}
	
	We next choose $T$ sufficiently large and $\delta$  sufficiently small so that 
	$$C\, (AT) H(T,\delta)^{\frac{1}{p}}\leq 1/2$$
	and thus deduce from the above bound and \eqref{limL} that 
	\begin{equation*}
	\tfrac{1}{2} L 	\leq 	\tfrac{1}{2} |B_0|^{\frac{1}{p}}  + C \norm{ {\bf f}/N}_{L^{p, \infty}(2B_0)}.
	\end{equation*}

	In view of  \eqref{SSS1} and \eqref{Neqn1} this gives 
	\begin{eqnarray}\label{fanaleF}
	\lefteqn{ \|\nabla u\|_{L^{p,\infty}(B_0)}  \leq C |B_0|^{\frac{1}{p}} N + C  \norm{{\bf f}}_{L^{p, \infty}(2B_0)}}\\
	&\leq&  C |B_0|^{\frac{1}{p}} \left(\frac{1}{|2B_0|} \int_{2B_{0}} |\nabla u|^{p-\delta} dx\right)^{\frac{1}{p-\delta}} + C  \norm{{\bf f}}_{L^{p, \infty}(2B_0)}.\nonumber
	\end{eqnarray}

	Finally, we write $\mu={\rm div}\, {\bf g}$ in $2B_0$, where $${\bf g}=-\nabla \int_{2B_0} G(x,y) d\mu(y)$$ 
	and $G(x,y)$ is the Green function associated with $-\Delta$ in $2B_0$. Note then that
	$$|{\bf g}|\leq C\, {\bf I}_{1}(\chi_{2B_0} \mu)$$
	and with ${\bf f}={\bf g} |{\bf g}|^{\frac{2-p}{p-1}}$ we have 		$|{\bf f}|^{p-2} {\bf f}={\bf g}$. Thus ${\rm div}\, (|{\bf f}|^{p-2} {\bf f})=\mu$
	in $2B_0$ and  
	\begin{equation}\label{sigtof}
	|{\bf f}|\leq C\, [{\bf I}_{1}(\chi_{2B_0}\mu)]^{\frac{1}{p-1}} .
	\end{equation}
	
	By \eqref{fanaleF}, this completes the proof of the theorem.
\end{proof}

We next prove a gradient estimate for solutions of \eqref{Basic-PDE} under condition \eqref{mucond}.

\begin{theorem}\label{Morreybound} Let  $1<p<n$, and let  $u$ be a nonnegative $p$-superharmonic solution of $-\Delta_p u=\mu$, where $\mu$ satisfies condition \eqref{mucond}. Then we have 
	\begin{equation*}%\label{MforR}
	\norm{\nabla u}_{L^{p, \infty}(B(x,R))} \leq C M^{\frac{1}{p-1}} 
	R^{\frac{n-p}{p}},  \quad \forall x\in\RR^n, \, R>0,
	\end{equation*}
	where 
	\begin{equation}\label{Mnorm}
	M =\sup_{x\in \RR^n, \, R>0} \frac{\mu(B(x,R))}{R^{n-p}}.
	\end{equation}
\end{theorem}

\begin{proof} Let $B_0=B(z_0,R_0)$ be any fixed  ball. By Theorem \ref{Lorentz-p-thick} we have 
	\begin{eqnarray}\label{B0bound2}
	\|\nabla u\|_{L^{p, \infty}(B_0))}&\leq& C\, |B_0|^{\frac{1}{p}}  \left(\frac{1}{|2B_0|}\int_{2B_0} |\nabla u|^{p-\delta} dx\right)^{\frac{1}{p-\delta}}   \\
	&& +\,  C\, \norm{[{\rm\bf I}_1(\chi_{2B_0}\mu)]^{\frac{1}{p-1}}}_{L^{p,\infty}(B_0)}\nonumber
	\end{eqnarray}
	for a constant  $\delta=\delta(n,p)\in (0, p-1)$ and we may assume that $\delta$ is sufficiently small.
	For any $r_0> 4R_0 + |z_0|$ and any  $r\in(0,r_0]$, let  $w\in u+ W^{1,\, \pmd}_{0}(B(z_0, r))$ solve
	\begin{equation*}
	\left\{ \begin{array}{rcl}
	\Delta_p w&=&0   \quad \text{in} ~B(z_0, r), \\ 
	w&=&u  \quad \text{on}~ \partial B(z_0,r).
	\end{array}\right.
	\end{equation*}
	
	By \cite[Lemma 2.7]{AP}   for any $0<\rho\leq r$ we have 
	\beas
	\label{holder-first-approx}
	\integral_{B(z_0,\rho)} |\nabla w|^{\pmd} \, dy \leq C  (\rho/r)^{n+(\pmd)(\beta_0-1)} \integral_{B(z_0, r)} |\nabla w|^{\pmd} \, dy,
	\enas
	for some  $\beta_0=\beta_0(n,p)\in(0,1/2]$. 
	Then by using \eqref{sigtof} and  arguing as in the proof of
	\cite[Equation (5.4)]{AP} we have

	\bea\label{iteration-first0}
	\phi(\rho) \leq C &\bigg[ \left( \frac{\rho}{r} \right)^{n+(\pmd)(\beta_0 -1)} + \delta^{(\pmd)\min\left\{1,\frac{1}{p-1}\right\}} +\ep \bigg] \phi(r)  \\
	& \qquad +\, C(\ep)   \,   \integral_{B(z_0, r)} {\bf I}_1(\chi_{B(z_0, r_0)}\mu )^{\frac{\pmd}{p-1}}\, dx,
	\ena
	which holds for all $\ep>0$  and $\rho\in (0, r]$. In  \eqref{iteration-first0}, we set
	$$\phi(\rho) = \integral_{B(z_0, \rho)} |\nabla u|^{\pmd}\, dx.$$

	Now by H\"older's inequality and \eqref{weakmorrey} we have 
	\begin{align*}
	\integral_{B(z_0, r)} {\bf I}_1(\chi_{B(z_0, r_0)}\mu )^{\frac{\pmd}{p-1}}\, dx &\leq C \norm{{\bf I}_1 \mu}_{L^{\frac{p}{p-1}, \infty}(B(z_0,r))}^{\frac{p-\delta}{p-1}} r^{\frac{n\delta}{p}}\\
	& \leq C M^{\frac{p-\delta}{p-1}} r^\frac{(n-p)(p-\delta)+n\delta }{p}\\
	& \leq C M^{\frac{p-\delta}{p-1}} r^{n-p+\delta},
	\end{align*}
	where $M$ is defined in \eqref{Mnorm}. 	Thus it follows from \eqref{iteration-first0} that 
	\beas
	\phi(\rho) \leq C &\bigg[ \left( \frac{\rho}{r} \right)^{n+(\pmd)(\beta_0 -1)} + \delta^{(\pmd)\min\left\{1,\frac{1}{p-1}\right\}} +\ep \bigg] \phi(r)  \\
	& \qquad +\, C(\ep)   \,   M^{\frac{p-\delta}{p-1}} r^{n-p+\delta},
	\enas
	which holds for all $\ep>0$  and $\rho\in (0, r]$. As $n-p+\delta< n+(\pmd)(\beta_0 -1)$,
	we can  apply  \cite[Lemma 3.4]{HL} to obtain  
	\beas
	\label{after-iteration}
	\phi(\rho) \leq C  \left( \frac{\rho}{r} \right)^{n-p+\delta}  \phi(r) +  C  M^{\frac{p-\delta}{p-1}} \rho^{n-p+\delta}  
	\enas
	provided   $\delta$ is sufficiently small. Since this estimate holds for all $0<\rho \leq  r \leq  r_0$, we may choose 
	$\rho = 2R_0$ and $r=r_0$ to deduce 
	\begin{align}
	\label{after-iteration-2}
	\integral_{B(z_0, 2R_0)} |\nabla u|^{\pmd}\, dx &\leq C  \left( \frac{R_0}{r_0} \right)^{n-p+\delta}  \integral_{B(z_0, r_0)} |\nabla u|^{\pmd}\, dx\\
	&  \qquad +\,  C  M^{\frac{p-\delta}{p-1}}  R_0^{n-p+\delta}\nonumber\\
	&\leq C  \left( \frac{R_0}{r_0} \right)^{n-p+\delta}  \integral_{B(0, 2r_0)} |\nabla u|^{\pmd}\, dx \nonumber\\
	&  \qquad  +\,  C  M^{\frac{p-\delta}{p-1}}  R_0^{n-p+\delta}, \nonumber
	\end{align}
	where we used that  $B(z_0,r_0)\subset B(0, 2r_0)$.
	
	At this point we combine \eqref{B0bound2}, \eqref{after-iteration-2}, and \eqref{weakmorrey}  to arrive at 
	\begin{eqnarray*}
		\|\nabla u\|_{L^{p, \infty}(B(z_0, R_0)))}&\leq& C\, R_0^{\frac{n-p}{p}}  \left( r_0^{-n+p-\delta}\int_{B(0, 2r_0)} |\nabla u|^{p-\delta} dx\right)^{\frac{1}{p-\delta}}   \\
		&& \qquad  +\,  C\, M^{\frac{1}{p-1}} R_0^{\frac{n-p}{p}}.
	\end{eqnarray*}
	
	Finally, letting $r_0\rightarrow \infty$ and applying 	Lemma \ref{decaylem} we complete the proof of the theorem.
\end{proof}	

We conclude this section with the following remarks regarding quasilinear equations with more general nonlinear structure.
\begin{remark}
	Theorems \ref{Lorentz-p-thick} and \ref{Morreybound} also hold for more general equations of the form 
	\begin{equation}\label{genstru}
	-{\rm div}\mathcal{A}(x, \nabla u)=\mu,
	\end{equation}	
	where $\aa: \RR^n\times\RR^n \rightarrow \RR^n$  is measurable in $x$ for every $\xi$, continuous in 
	$\xi$ for a.e. $x$, and  $\aa(x,0)=0$ for a.e. $x\in \RR^n$. Moreover, $\aa$ is assumed to satisfy that 
	\begin{equation}\label{monotone}
	\langle\aa(x,\xi)-\aa(x,\zeta),\xi-\zeta \rangle\geq
	\Lambda_0 (|\xi|^2+|\zeta|^2)^{\frac{p-2}{2}}|\xi-\zeta|^2 
	\end{equation}
	and for some $\gamma\in(0,1)$,
	\begin{align}
	\label{ellipticity}
	|\aa(x,\xi) - \aa(x,\zeta)| \leq \Lambda_1 |\xi - \zeta|^{\gamma} (|\xi|^2+|\zeta|^2)^{\frac{p-1 - \gamma}{2}}
	\end{align}
	for every $(\xi,\zeta)\in\RR^n \times\RR^n\setminus\{(0,0)\}$ and a.e. $x \in \RR^n$. Here $\Lambda_0$ and $\Lambda_1$ are  positive constants. 
\end{remark}

\begin{remark} 
	Condition \eqref{ellipticity} above can be  replaced with the weaker condition
	\begin{align}
	\label{ellipticity22}
	|\aa(x,\xi)| \leq \Lambda_1 |\xi|^{p-1}.
	\end{align}
\end{remark}

Indeed, for $\frac{3n-2}{2n-1}<p<n$,  this can be done similarly using the method of \cite{Ph} and the  comparison estimate of \cite[Lemma 2.2]{NP1} (see also \cite{Min2}, where this method was first utilized in the case $p\geq 2$). For $1<p\leq \frac{3n-2}{2n-1}$,  using the method of \cite{Ph} and the recent comparison estimate of \cite[Lemma 2.1]{NP2},  one can obtain the following version of \eqref{B0bound}: 

\textit{There exists $\epsilon_0=\epsilon_0(n,p,\Lambda_0, \Lambda_1)\in (0, 2(p-1))$
	such that  for $2-p+\epsilon_0<q<p+\epsilon_0$,
	\begin{eqnarray}\label{local-q}
	\|\nabla u\|_{L^{q, \infty}(B_0))}&\leq& C(\epsilon)\, |B_0|^{\frac{1}{q}}  \left(\frac{1}{|2B_0|}\int_{2B_0} |\nabla u|^{2-p} dx\right)^{\frac{1}{2-p}}   \\
	&& +\,  C(\epsilon)\, \norm{[{\rm\bf I}_1(\chi_{2B_0}\mu)]^{\frac{1}{p-1}}}_{L^{q, \infty}(B_0)}\nonumber\\
	&& +\, \epsilon\, \|\nabla u\|_{L^{q, \infty}(2B_0)}  \nonumber
	\end{eqnarray} 
	for all balls $B_0$ and all $\epsilon>0$.} 

(This estimate still holds if the weak $L^q$ norms are replaced with the  $L^q$ norms). The constant $C(\epsilon)$ is independent of $q$.
Thus, for  $2-p+\epsilon_0<q<p$,  by Lemma \ref{decaylem} and a covering/iteration argument    (see, e.g., \cite{AMP}) the term $\epsilon\, \|\nabla u\|_{L^{q, \infty}(2B_0)}$ on the right-hand side can be absorbed yielding that 
\begin{eqnarray}\label{qlessp}
\|\nabla u\|_{L^{q, \infty}(B_0))}&\leq& C\, |B_0|^{\frac{1}{q}}  \left(\frac{1}{|2B_0|}\int_{2B_0} |\nabla u|^{2-p} dx\right)^{\frac{1}{2-p}}   \\
&& +\,  C\, \norm{[{\rm\bf I}_1(\chi_{2B_0}\mu)]^{\frac{1}{p-1}}}_{L^{q, \infty}(B_0)}  \nonumber
\end{eqnarray} 
for all balls $B_0$. Thus letting $q\uparrow p$ we see that \eqref{qlessp}  holds with $q=p$ as well.
From this we obtain analogues of Theorems \ref{Lorentz-p-thick} and \ref{Morreybound} under the above assumptions on $\mathcal{A}$.

Using Poincar\'e's inequality we deduce the following BMO estimate. 

\begin{corollary}	Let $1<p<n$, and let $\mu$ satisfy condition \eqref{mucond}. Under assumptions 
	\eqref{monotone} and  \eqref{ellipticity22} on $\mathcal{A}$, 
	for any nonnegative 
	$\mathcal{A}$-superharmonic solution $u$ to \eqref{genstru} we have 
	\begin{equation*}
	\norm{u}_{{\rm BMO}(\RR^n)} \leq C M^{\frac{1}{p-1}}, 
	\end{equation*}
	where $M$ is the constant in \eqref{Mnorm}, and $C$ depends only 
	on $p, n, \Lambda_0, \Lambda_1$. 
\end{corollary}

\begin{remark} In the case $0<q<p-1$, Theorem \ref{thm3} and Corollary \ref{cor-1} 
	are deduced exactly as in \cite{V1} using Theorems \ref{thm1}, \ref{Lorentz-p-thick} and \ref{Morreybound} in place of the corresponding statements of \cite[Lemma 3.1]{V1}.  
\end{remark}

\section{Proof of Theorem \ref{thm2}}\label{sec3}

In this section we treat the case $q>p-1$ in \eqref{pq-PDE}. Let $1<p<n$.  As was shown in \cite{PV2}, the existence of a solution $u$ to \eqref{pq-PDE} is equivalent to condition (a) of Theorem \ref{thm2}, 
with the small constant $0<c \le c(n, p, q)$ in the sufficiency part, and some $c>0$ 
in the  necessity part, where ${\bf W}_p \mu \not\equiv\infty$. Let $d \omega =u^q d\sigma + d\mu$. 
By Theorem \ref{thm1}, any solution $u$ to \eqref{pq-PDE} lies in 
${\rm BMO}(\RR^n)$
if and only if 
\begin{equation}\label{omega-est}
\omega (B(x, R)) \le C \, R^{n-p}, \quad \forall x \in \RR^n, \, R>0.
\end{equation}
In particular, $\mu$ satisfies \eqref{mucond}. Also, by the lower bound in \eqref{K-M}, 
we have $u \ge C \, {\bf W}_p \mu$, so that by \eqref{mucond} 
\begin{equation}\label{Wmu-est}
\int_{B(x, R)} ({\bf W}_p \mu)^q d \sigma  \le C \, R^{n-p}, \quad \forall x \in \RR^n, \, R>0.
\end{equation}
This yields the necessity of condition (b) in  Theorem \ref{thm2},   since for 
all $y\in B(x, R)$ and $r>R$, we have $B(x, r)\subset B(y, 2r)$, and consequently 
\begin{align*}
{\bf W}_p \mu(y) = & 2^{-\frac{n-p}{p-1}} \,  \int_{0}^\infty \left( \frac{\mu(B(y, 2r))}{r^{n-p}}\right)^{\frac{1}{p-1}}
\frac{dr}{r} \\ \ge & 2^{-\frac{n-p}{p-1}} \,  \int_{R}^\infty \left( \frac{\mu(B(x, r))}{r^{n-p}}\right)^{\frac{1}{p-1}}
\frac{dr}{r}.  
\end{align*}

Conversely, suppose that \eqref{mucond}, and both condition (a) with the small constant $c \le c(n, p, q)$, and condition (b)  of Theorem \ref{thm2} hold. Then the  solution $u$ constructed in \cite[Theorem 3.10]{PV2} admits the upper bound $u \le C \, {\bf W}_p \mu$. Hence, to verify \eqref{omega-est}, it remains to show that \eqref{Wmu-est} holds. 

For $B=B(x, R)$, we write $\mu=\mu_{2B} + \mu_{(2B)^c}$. Then clearly 
$$
{\bf W}_p \mu \le c \Big( {\bf W}_p \mu_{2B} + {\bf W}_p \mu_{(2B)^c} \Big),
$$
where $c$ depends only on $p$. Arguing as above, for all $y\in B(x, R)$, we have 
$B(y,r)\cap (2B)^c=\emptyset$ if $0<r<R$, and 
$B(y,r)\cap (2B)^c\subset B(x, 2r))$ for $r\ge R$, so that
$$
{\bf W}_p \mu_{(2B)^c}(y) \le C \,  \int_{R}^\infty \left( \frac{\mu(B(x, r))}{r^{n-p}}\right)^{\frac{1}{p-1}}
\frac{dr}{r}.    
$$
Hence by condition (b)  of Theorem \ref{thm2}, we see that \eqref{Wmu-est} holds for 
$\mu_{(2B)^c}$ in place of $\mu$. Also, as was shown in \cite[Theorem 3.1]{PV2} 
(with $g=\chi_{2B}$), condition (a) of Theorem  \ref{thm2} yields 
$$
\int_B ( {\bf W}_p \mu_{2B})^q d \sigma \le C \, \mu(2B).  
$$
Since $\mu(2B) \le C \, R^{n-p}$ by  \eqref{mucond}, combining the preceding 
estimates we deduce \eqref{Wmu-est}. This completes the proof of Theorem \ref{thm2}. \qed

\section{Proof of Theorem \ref{thm3}} %\label{sec4}

In this section we treat the case $0<q<p-1$ in \eqref{pq-PDE}. Let $1<p<n$. It was proved in \cite{CV} (see also \cite{V2}) that a nontrivial solution to \eqref{pq-PDE} exists if and only if \eqref{finiteness} holds, i.e., ${\bf W}_p \mu\not\equiv\infty$, and 
\begin{equation}\label{cond-K}
\int_1^{\infty} \left(\frac{ \kappa(B(0, r))^{\frac{q(p-1)}{p-1-q}}}{r^{n- p}}\right)^{\frac{1}{p-1}}\frac{dr}{r} < \infty.  
\end{equation}

Condition \eqref{cond-K} ensures that ${\bf K}_{p, q}\sigma\not\equiv\infty$, where 
${\bf K}_{p, q}$ is the so-called  \textit{intrinsic} nonlinear potential introduced in \cite{CV}, 
\begin{equation*}
{\bf K}_{p, q}  \sigma (x)  =  \int_0^{\infty} \left(\frac{ \kappa(B(x, r))^{\frac{q(p-1)}{p-1-q}}}{r^{n- p}}\right)^{\frac{1}{p-1}}\frac{dr}{r}, \quad x \in \RR^n.
\end{equation*} 
Moreover, as was proved recently in \cite{V2}, any nontrivial solution $u$ to \eqref{pq-PDE} satisfies the bilateral estimates 
\begin{equation} \label{bilateral}
\begin{aligned}
c^{-1} & \left [({\bf W}_p \sigma(x))^{\frac{p-1}{p-1-q}}+ {\bf K}_{p, q}  \sigma(x) 
+ {\bf W}_p \mu(x)\right] 
\le u (x) \\ &\le c \left[({\bf W}_p \sigma(x))^{\frac{p-1}{p-1-q}} + {\bf K}_{p, q}  \sigma(x) 
+ {\bf W}_p \mu(x)\right],  \quad x\in \RR^n, 
\end{aligned}
\end{equation}
where $c>0$ is a constant which  depends only on $p$, $q$, and $n$. 

As in the case $q>p-1$, by Theorem \ref{thm1}, any solution $u$ to \eqref{pq-PDE} lies in 
${\rm BMO}(\RR^n)$
if and only if \eqref{omega-est} holds, where $d \omega =u^q d\sigma + d\mu$. 
In view of \eqref{bilateral}, $u \in {\rm BMO}(\RR^n)$ if and only 
if  both conditions \eqref{mucond} and \eqref{Wmu-est} hold, and also the following two conditions hold 
for all $x \in \RR^n$ and $R>0$:
\begin{equation}\label{Wsigma-pq}
\int_{B(x, R)} ({\bf W}_p \sigma)^{\frac{q(p-1)}{p-1-q}} d \sigma  \le C \, R^{n-p}, \end{equation}
\begin{equation}\label{Ksigma-est}
\int_{B(x, R)} ({\bf K}_{p, q} \sigma)^q d \sigma  \le C \, R^{n-p}. 
\end{equation}

We first show that \eqref{mucond} together with conditions (a)--(d) of Theorem \ref{thm3} yield \eqref{Wmu-est}, \eqref{Wsigma-pq}, and \eqref{Ksigma-est}.

As in the case $q>p-1$ above, \eqref{Wmu-est} splits into two parts: condition (b) of Theorem \ref{thm3}, and 
\begin{equation}\label{Wsigma-mu}
\int_{B} ( {\bf W}_p \mu_{2B})^q d \sigma \le C \, R^{n-p},  
\end{equation} 
where $B=B(x, R)$. To prove the preceding estimate, notice that by \eqref{bestcst} 
applied to $\nu=\mu_{2B}$, we have 
$$
\int_{B} ( {\bf W}_p \mu_{2B})^q d \sigma \le \kappa(B)^q \, \mu(2B)^\frac{q}{p-1}.
$$
By \eqref{mucond}, it follows that $\mu(2B)\le C \, R^{n-p}$, and  by condition (a), we have $\kappa(B)^q \le C \, R^{\frac{(n-p)(p-1-q)}{p-1}}$. 
Hence,  \eqref{Wsigma-mu} follows from \eqref{mucond}\&(a), and consequently 
\eqref{Wmu-est}  follows from (a)\&(b)\&\eqref{mucond}. 

To prove \eqref{Wsigma-pq}, for $B=B(x, R)$, we write 
$\sigma=\sigma_{2B} + \sigma_{(2B)^c}$. Again, \eqref{Wsigma-pq} splits into two 
parts. Arguing as above in the case $q>p-1$  we have 
\begin{align*}
\int_{B} ({\bf W}_p \sigma_{(2B)^c})^{\frac{q(p-1)}{p-1-q}} d \sigma & \le C \, 
\sigma(B) \, \left[\int_R^{\infty} \left(\frac{ \sigma(B(x, r))^{\frac{q(p-1)}{p-1-q}}}{r^{n- p}}\right)^{\frac{1}{p-1}}\frac{dr}{r}\right]^{\frac{q(p-1)}{p-1-q}} \\& \le C \, R^{n-p}, 
\end{align*}
by condition (c) of Theorem \ref{thm3}. 

Next, denote by $v_{2B}\in L^q(\sigma_{2B})$ the nontrivial solution to the equation 
\begin{equation}\label{v2b}
v_{2B} = {\bf W}_p (v^q_{2B}\sigma_{2B}) \quad \text{in} \, \, \RR^n
\end{equation} 
constructed in \cite{CV}, which exists since $\kappa(2B)<\infty$ by condition (a) of Theorem \ref{thm3}. By 
\cite[Corollary 4.3]{CV}, 
\begin{equation}\label{cor4.3}
\int_{2B} (v_{2B})^q d \sigma \le \kappa(2B)^{\frac{q(p-1)}{p-1-q}}.   
\end{equation} 
On the other hand, $v_{2B}\ge C \,  {\bf W}_p \sigma_{2B})^{\frac{p-1}{p-1-q}}$ 
by the lower estimate in \eqref{bilateral}. Combining these estimates yields 
$$
\int_{B} ({\bf W}_p \sigma_{2B})^{\frac{q(p-1)}{p-1-q}} d \sigma 
\le  C \kappa(2B)^{\frac{q(p-1)}{p-1-q}}\le C R^{n-p}
$$
by condition (a). 

We now prove \eqref{Ksigma-est}. For $y\in B=B(x, R)$, we split 
${\bf K}_{p, q} \sigma(y)$ into two parts, 
\begin{align*}
{\bf K}_{p, q} \sigma(y)  = I + II =
& \int_0^{R} \left(\frac{ \kappa(B(y, r))^{\frac{q(p-1)}{p-1-q}}}{r^{n- p}}\right)^{\frac{1}{p-1}}\frac{dr}{r} \\& + \int_R^{\infty} \left(\frac{ \kappa(B(y, r))^{\frac{q(p-1)}{p-1-q}}}{r^{n- p}}\right)^{\frac{1}{p-1}}\frac{dr}{r}. 
\end{align*}
To estimate the term involving $I$, notice that $B(y, r) \subset 2B$ for $0<r\le R$. Hence by 
the lower estimate in \eqref{bilateral} with $\mu=0$ and $\sigma_{2B}$ in place of 
$\sigma$, we have $I \le C \, v_{2B}$, where $v_{2B}$ is defined by \eqref{v2b}. 
It follows that 
$$
\int_{B} I^q d \sigma \le C \, \int_{B} v^q_{2B} d \sigma. 
$$
By the preceding estimate, \eqref{cor4.3}, and condition (a), we deduce 
$$
\int_{B} I^q d \sigma  \le C \,\kappa(2B)^{\frac{q(p-1)}{p-1-q}} \le C \, R^{n-p}. 
$$

For $r>R$ and $y\in B$, we obviously have $B(y, r) \subset B(x, 2r)$, so that 
$\kappa(B(y, r))\le \kappa(B(x, 2r))$, and consequently, for all $y\in B$,  
$$
II\le 2^{\frac{n-p}{p-1}} \, \int_{2R}^{\infty} \left(\frac{ \kappa(B(x, r))^{\frac{q(p-1)}{p-1-q}}}{r^{n- p}}\right)^{\frac{1}{p-1}} \frac{dr}{r}. 
$$
It follows that
$$
\int_{B} II^q d \sigma \le C \, \sigma(B) \,  \left[\int_{2R}^{\infty} \left(\frac{ \kappa(B(x, r))^{\frac{q(p-1)}{p-1-q}}}{r^{n- p}}\right)^{\frac{1}{p-1}} \frac{dr}{r}\right]^q\le C \, R^{n-p}
$$
by condition (d) of Theorem \ref{thm3}. This proves that \eqref{Ksigma-est} 
follows from conditions (a)\&(d). Thus, \eqref{omega-est} holds, so that 
$u\in {\rm BMO}(\RR^n)$. 

Conversely, suppose that  $u \in {\rm BMO}(\RR^n)$ is a solution to 
\eqref{pq-PDE}. Then as was mentioned above \eqref{omega-est} holds, 
which obviously yields \eqref{mucond}. Since \eqref{omega-est} also yields  
\begin{equation}\label{uq-est}
\int_B u^q d \sigma \le C \, R^{n-p}, 
\end{equation}
and by \cite[Lemma 4.2]{CV},   
\begin{equation}\label{kappa-uq}
\kappa(B)^{\frac{q(p-1)}{p-1-q}}\le C \,  \int_B u^q d \sigma, 
\end{equation}
we combine \eqref{uq-est} and \eqref{kappa-uq} to obtain (a).

Next, by \eqref{uq-est} and the lower estimate in \eqref{bilateral} we deduce that 
\eqref{Wmu-est}, 
\eqref{Wsigma-pq}, and \eqref{Ksigma-est} hold. 

Notice that condition (b) follows from \eqref{Wmu-est} exactly as in the case 
$q>p-1$ above. Similarly, for all $y \in B=B(x, R)$, we have 
$$
{\bf W}_p \mu(y)  \ge 2^{-\frac{n-p}{p-1}} \,  \int_{R}^\infty \left( \frac{\mu(B(x, r))}{r^{n-p}}\right)^{\frac{1}{p-1}}
\frac{dr}{r}.  
$$
Hence, \eqref{Wsigma-pq} yields condition (c). In the same way, for all $y \in B=B(x, R)$ and $r>R$, we have $B(y, 2r)\supset B(x, r)$, and consequently 
\begin{align*}
{\bf K}_{p, q} \sigma(y) & =2^{-\frac{n-p}{p-1}} \,
\int_0^{\infty} \left(\frac{ \kappa(B(y, 2r))^{\frac{q(p-1)}{p-1-q}}}{r^{n- p}}\right)^{\frac{1}{p-1}}\frac{dr}{r} \\ & \ge 2^{-\frac{n-p}{p-1}}\int_R^{\infty} \left(\frac{ \kappa(B(x, r))^{\frac{q(p-1)}{p-1-q}}}{r^{n- p}}\right)^{\frac{1}{p-1}}
\frac{dr}{r}. 
\end{align*}
This shows that \eqref{Ksigma-est} yields condition (d). The proof of 
Theorem \ref{thm3} is complete. \qed

The proof of Corollary \ref{cor-1} is based on the following pointwise estimate 
for all solutions $u$ to \eqref{pq-PDE} in the case $0<q<p-1$ \cite[Corollary 1.2]{V2},  
\begin{equation*} 
\begin{aligned}
c^{-1} & \left [({\bf W}_p \sigma(x))^{\frac{p-1}{p-1-q}} 
+ {\bf W}_p \mu(x)\right] 
\le u (x) \\ &\le c \left[({\bf W}_p \sigma(x))^{\frac{p-1}{p-1-q}} +  {\bf W}_p \sigma(x) 
+ {\bf W}_p \mu(x)\right],  \quad x\in \RR^n, 
\end{aligned}
\end{equation*}
provided  $\sigma$ satisfies condition \eqref{cap-K}. The argument is similar to that of \cite[Corollary 1.5]{V1} in 
the case $\mu=0$; we omit the details. 

\section{The natural growth case}\label{sec5}

In this section we suppose that $1<p<n$  and 
$\mu, \sigma \in M^{+}(\RR^n)$, where both $\sigma\not=0$ and $\mu\not=0$. It is well known   (see, for instance, \cite{JV}) 
that the capacity condition \eqref{cap-K} with $C=1$ 
is  \textit{necessary}    for the existence 
of a nontrivial solution $u$ to the inequality 
\begin{equation*}
-\Delta_p u \ge \sigma \, u^{p-1}, \quad u\ge 0 \quad \text{in} \, \, \RR^n. 
\end{equation*}

We have to distinguish between 
the cases $p> 2$ and $p\le 2$. We recall that $\mathbf{I}_p$ stands for the Riesz 
potential of order $p$ defined by \eqref{riesz} with $\alpha=p$. It is easy to see that 
\begin{equation}\label{cond-p}
{\bf I}_{p} \sigma \le C \,  (\mathbf{W}_p \sigma)^{p-1} \quad  \textrm{if} \, \, p> 2,  \quad \textrm{and} \quad  
({\bf W}_{p} \sigma)^{p-1} \le C \,  \mathbf{I}_p \sigma \quad  \textrm{if} \, \, p\le 2,  
\end{equation}
where $C$ is a constant which depends only on $p$ and $n$.

\begin{theorem}\label{thm4} Let   $1<p<n$ and $q=p-1$. Suppose 
	$\mu, \sigma\in M_{+}(\RR^n)$,  and 
	\begin{align*}
	& (a) \quad  {\bf W}_{p} \sigma \le C_1  \quad \text{if} \, \, \,  p> 2, \qquad (b) \quad 
	\mathbf{I}_p \sigma \le C_2 \quad \text{if} \, \, \,  p\le 2.
	\end{align*}
	Then there exists a  solution $u \in {\rm BMO}(\RR^{n})$ 
	to \eqref{pq-PDE} if and only if $\mu$ satisfies condition \eqref{mucond}, and 
	for all $x \in \RR^n$, $R>0$, 
	\begin{equation}\label{cond-b}
	\quad \sigma (B(x, R))  \left[ \int_{R}^\infty \left( \frac{\mu(B(x, r))}{r^{n-p}}\right)^{\frac{1}{p-1}}
	\frac{dr}{r}\right]^{p-1}\le C \, R^{n-p},
	\end{equation}
	where the ``if'' part requires the smallness  of the constant $c=c(p, n)$ in the 
	condition 
	\begin{equation}\label{small-const}
	\sigma (K) \le c\, {\rm cap}_p(K), \quad \forall \, {\rm compact  \,  sets} \,  K\subset \RR^n.
	\end{equation}
\end{theorem}

\begin{remark} Assumptions {\rm (a)} and {\rm (b)} in Theorem \ref{thm4} 
	are stronger than the necessary condition  \eqref{small-const}  
	for some constant $c$. Without these assumptions,  
	estimates of solutions  are substantially more complicated (see \cite{JV}). 
\end{remark} 

\begin{proof} It is known (\cite[Remark 1.3 and Sec. 2]{JV} that conditions (a)\&(b) of Theorem \ref{thm4}, together with 
	\eqref{small-const} for some small  constant $c=c(p, n)$,  
	ensure that \eqref{pq-PDE} has a solution $u$ such that 
	\begin{equation}\label{p-1-two-sided}
	c_1 \, {\bf W}_p \mu(x) \le u(x)\le  c_2 \, {\bf W}_p \mu(x), \quad x \in \RR^n.
	\end{equation}
	The lower bound obviously holds for all solutions $u$. 
	
	As above,  by Theorem \ref{thm1}, 
	$u \in {\rm BMO}(\RR^n)$ if and only if $\mu$ satisfies \eqref{mucond}, and \eqref{uq-est} holds (with $q=p-1$). By the lower estimate in \eqref{p-1-two-sided}, we see that \eqref{uq-est} 
	yields 
	\begin{equation}\label{p-1-mu-sigma}
	\int_{B(x, R)} ({\bf W}_p \mu)^{p-1} d \sigma \le C \, R^{n-p}, \quad \forall \, x \in \RR^n, 
	\, R>0. 
	\end{equation}

	Exactly as  in the cases $q>p-1$ and $q<p-1$, this estimate yields 
	\eqref{cond-b}, which completes the proof of the ``only if'' part of 
	Theorem \ref{thm4}.

	To prove the ``if'' part, as above we split \eqref{p-1-mu-sigma} into two parts, 
	condition \eqref{cond-b} and 
	\begin{equation}\label{p-1-mu-sigma-loc}
	\int_{B} ({\bf W}_p \mu_{2B})^{p-1} d \sigma \le C \, \, |B|^{\frac{n-p}{n}},      \end{equation}
	where  $B=B(x, R)$. 
	
	We first prove \eqref{p-1-mu-sigma-loc} in the easier case $1<p\le 2$. It follows from     \eqref{cond-p} that $({\bf W}_p \mu_{2B})^{p-1} \le C \, {\bf I}_p \mu_{2B} $, 
	and by Fubini's Theorem, 
	$$
	\int_{B} ({\bf W}_p \mu_{2B})^{p-1} d \sigma \le C  \int_{B} {\bf I}_p \mu_{2B} \,  d \sigma =C  \int_{2B} {\bf I}_p \sigma_{B} \,  d \mu.
	$$
	Since ${\bf I}_p \sigma_{B}\le C_2$ by assumption (b), we deduce 
	$$
	\int_{B} ({\bf W}_p \mu_{2B})^{p-1} d \sigma \le C \, C_2 \mu(2B),
	$$
	and \eqref{p-1-mu-sigma-loc} follows in view of condition \eqref{mucond}. 
	
	We now consider the case $p>2$. Then \eqref{p-1-mu-sigma-loc} can be deduced 
	from \cite[Lemma 4.4]{JV}, but we 
	give here a simplified proof based on the following lemma.
	
	\begin{lemma}\label{lemma5.2} Let $2<p<n$, and let $\mu, \sigma\in M^{+}(\RR^n)$, 
		where $\sigma$ satisfies \eqref{small-const}. Then 
		\begin{equation}\label{lemma-est}
		\int_{\RR^n} ({\bf W}_p \mu)^{p-1} d \sigma \le C \,  c^{\frac{p-2}{p-1}}
		\int_{\RR^n} ({\bf W}_p \sigma) \,  d \mu,
		\end{equation} 
		where $c$ is the constant  in \eqref{small-const}, and  $C$ is a constant which  depends only on $p, n$. 
	\end{lemma}
	
	\begin{proof} It is more convenient to use  dyadic Wolff 
		potentials introduced originally in \cite{HeWo}, in place of ${\bf W}_p \mu$,  
		\begin{equation*}
		{\bf W}^d_p \mu(x) =\sum_{Q \in \mathcal{D}}
		\left ( \frac{\mu(Q)}{ \ell(Q)^{n-p}}\right)^{\frac{1}{p-1} 
		} \chi_Q(x), \quad 
		x \in \RR^n, 
		\end{equation*} 
		where  $\mathcal{D}=\{ Q\}$ is the family of all dyadic cubes in $\RR^n$, and 
		$\ell(Q)$ stands for the side length of $Q$. 
		
		For $Q \in  \mathcal{D}$, we denote by 
		$Q^*$  the concentric cube with side length $\ell(Q^*)=3 \, \ell(Q)$.  
		Clearly, the family of cubes $\{Q^*\}_{Q \in \mathcal{D}}$ has the finite intersection property
		\begin{equation}\label{finite-inter}
		\sum_{\ell(Q)=2^k} \chi_{Q^*}(x) \le \beta(n), \qquad x\in \RR^n, \, \, k\in \mathbb{Z}, 
		\end{equation}
		where $\beta(n)$ is a constant which depends only on $n$. We will actually need  
		a modified version of ${\bf W}^d_p$ defined by 
		\begin{equation*}
		\widetilde{{\bf W}}^d_p \mu = \sum_{Q \in \mathcal{D}}
		\left ( \frac{\mu(Q^*)}{ \ell(Q)^{n-p}}\right)^{\frac{1}{p-1} 
		} \chi_Q(x), \quad 
		x \in \RR^n. 
		\end{equation*} 
		
		It is easy to verify (see \cite[p. 170]{HeWo}) that 
		\begin{equation}\label{d-wolff-A}
		a \,  {\bf W}^d_p \mu \le {\bf W}_p \mu \le A \, \widetilde{{\bf W}}^d_p \mu,
		\end{equation} 
		where the constants $a, A$ depend only on $p$ and $n$.

		In view of \eqref{d-wolff-A}, 
		it is enough to prove the following version of \eqref{lemma-est}, 
		\begin{equation}\label{d-mu-sigma-loc}
		\int_{\RR^n} (\widetilde{{\bf W}}^d_p \mu)^{p-1} d \sigma  \le C \, c^{\frac{p-2}{p-1}} \int_{\RR^n} ({\bf W}_p \sigma) \,  d \mu.   
		\end{equation}

		Since $p>2$, we can  use duality to rewrite  \eqref{d-mu-sigma-loc} in the equivalent form 
		\begin{equation}\label{equiv-form}
		\int_{\RR^n} (\widetilde{{\bf W}}^d_p \mu) \,  g \, d \sigma \le 
		C \, c^{\frac{p-2}{(p-1)^2}}    \left[\int_{\RR^n} ({\bf W}_p \sigma) \,  d \mu\right]^{\frac{1}{p-1}}, 
		\end{equation}
		for all $g \in L^{\frac{p-1}{p-2}} (\RR^n, \sigma)$ such that $||g||_{L^{\frac{p-1}{p-2}} (\RR^n, \sigma)}\le 1$. Interchanging the order of integration and summation 
		on the left-hand side of \eqref{equiv-form}, we see that 
		\eqref{d-mu-sigma-loc}  is equivalent to 
		\begin{equation*}
		I= \sum_{Q \in \mathcal{D}}
		\left ( \frac{\mu(Q^*)}{\ell(Q)^{n-p} }\right)^{\frac{1}{p-1}} \int_Q g \, 
		d \sigma 
		\le  C \, c^{\frac{p-2}{(p-1)^2}} \,  \left[\int_{\RR^n} ({\bf W}_p \sigma) \,  d \mu\right]^{\frac{1}{p-1}}.  
		\end{equation*}
		Using H\"{o}lder's inequality with exponents $p-1$ and $\frac{p-1}{p-2}$, we estimate 
		\begin{equation} \label{holder}
		\begin{aligned}
		I \le &\left[  \sum_{Q \in \mathcal{D}}
		\left ( \frac{\sigma(Q)}{\ell(Q)^{n-p}}\right)^{\frac{1}{p-1}}   \mu(Q^*) 
		\right]^{\frac{1}{p-1}} \\
		& \times \left[  
		\sum_{Q \in \mathcal{D}} \frac{ \sigma(Q)^{p'} }{\ell(Q)^{\frac{n-p}{p-1}}}
		\left ( \frac{1}{\sigma(Q)} \int_Q g \, d \sigma  \right) ^{\frac{p-1}{p-2}}
		\right]^{\frac{p-2}{p-1}}, 
		\end{aligned}
		\end{equation}
		where $p'=\frac{p}{p-1}$. Notice that 
		$$
		\sum_{Q \in \mathcal{D}} \left ( \frac{\sigma(Q)}{\ell(Q)^{n-p} }\right)^{\frac{1}{p-1}}   \mu(Q^*) 
		=\int_{\RR^n} \sum_{Q \in \mathcal{D}} \left ( \frac{\sigma(Q)}{\ell(Q)^{n-p}}\right)^{\frac{1}{p-1}} \chi_{Q^*} \, d \mu.  
		$$
		If $x \in Q^*$, then obviously $Q \subset B(x, \alpha(n) \, \ell(Q))$, where   $\alpha(n)$ is a constant which depends only on $n$. We estimate  
		\begin{align*}
		\sum_{Q \in \mathcal{D}} \left ( \frac{\sigma(Q)}{\ell(Q)^{n-p}}\right)^{\frac{1}{p-1}} \chi_{Q^*}(x)&=\sum_{k\in \mathbb{Z}} \sum_{\ell(Q)=2^k}\left ( \frac{\sigma(Q)}{\ell(Q)^{n-p}}\right)^{\frac{1}{p-1}} \chi_{Q^*}(x)
		\\ & \le \sum_{k\in \mathbb{Z}} \, \left(\frac{\sigma(B(x, \alpha(n) 2^{k}))}    {2^{k(n-p)}}\right)^{\frac{1}{p-1}} 
		\sum_{\ell(Q)=2^k}\chi_Q^* (x).   
		\end{align*}
		Clearly,
		\begin{align*}
		\sum_{k\in \mathbb{Z}} \, \left(\frac{\sigma(B(x, \alpha(n) 2^{k}))}    {2^{k(n-p)}}\right)^{\frac{1}{p-1}} 
		& \le C \int_0^\infty \left(\frac{\sigma(B(x, \alpha(n) r)}{r^{n-p}} \right)^{\frac{1}{p-1}} \frac{dr}{r} 
		\\ &= C \, \alpha(n)^{\frac{n-p}{p-1}}  {\bf W}_p \sigma(x), 
		\end{align*}
		where $C$ depends only on $p$ and $n$. Hence, by the finite intersection property 
		\eqref{finite-inter}, 
		$$
		\sum_{Q \in \mathcal{D}} \left ( \frac{\sigma(Q)}{\ell(Q)^{n-p}}\right)^{\frac{1}{p-1}} \chi_{Q^*}(x) \le C\, \alpha(n)^{\frac{n-p}{p-1}} \, \beta(n)  {\bf W}_p \sigma(x). 
		$$
		Integration both sides of the preceding inequality with respect to $d \mu$ gives 
		$$
		\sum_{Q \in \mathcal{D}} \left ( \frac{\sigma(Q)}{\ell(Q)^{n-p} }\right)^{\frac{1}{p-1}} \mu(Q^*) \le C  \, \alpha(n)^{\frac{n-p}{p-1}} \, \beta(n) \int_{\RR^n} ({\bf W}_p \sigma) \,  
		d \mu. 
		$$

		We estimate the second factor in \eqref{holder} using the dyadic Carleson measure theorem.  We observe that assumption (a) yields 
		the capacity condition 
		\eqref{small-const}. It  
		is known (see \cite[Theorem 3.9]{JV}) that  \eqref{small-const} is equivalent to the dyadic Carleson measure 
		condition 
		$$
		\sum_{Q \subseteq P} \frac{ \sigma(Q)^{p'} }{\ell(Q)^{ \frac{n-p}{p-1} } } \le C \, c^{p'-1} \sigma(P),
		$$
		for all dyadic cubes $P$, where $c$ is the constant 
		in \eqref{small-const}, and 
		$C$ depends only on $p, n$. Hence by the dyadic Carleson measure theorem, 
		$$
		\sum_{Q \in \mathcal{D}} \frac{ \sigma(Q)^{p'} }{ \ell(Q)^{ \frac{n-p}{p-1} }} 
		\left ( \frac{1}{\sigma(Q)} \int_Q g \, d \sigma  \right) ^{\frac{p-1}{p-2}} \le 
		C \,  c^{p'-1} \, ||g||^{\frac{p-1}{p-2}}_{L^{\frac{p-1}{p-2}} (\RR^n, \sigma)}\le C \, c^{p'-1},  
		$$
		since $||g||_{L^{\frac{p-1}{p-2}} (\RR^n, \sigma)}\le 1$. Combining the preceding estimates proves \eqref{lemma-est}. 
		
	\end{proof}

	Applying Lemma \ref{lemma5.2} with $ \mu_{2B}$ and $\sigma_B$ in place of 
	$\mu$ and $\sigma$, respectively, we obtain 
	$$
	\int_B ({\bf W}_p \mu_{2B})^{p-1} d \sigma \le C \,  c^{\frac{p-2}{p-1}} \int_{2B} ({\bf W}_p \sigma_{B}) d \mu.  
	$$
	Invoking  assumption (a) and condition  \eqref{mucond} yields 
	$$
	\int_{2B} ({\bf W}_p \sigma_{B}) d \mu\le C \, C_1 \,  c^{\frac{p-2}{p-1}} \mu(2B)\le C \, |B|^{\frac{n-p}{n}}. 
	$$
	Thus, \eqref{p-1-mu-sigma-loc} holds for all $1<p<n$,  and consequently 
	$u \in {\rm BMO}(\RR^n)$.
\end{proof}

\bibliography{samplebib3_for_arXiv}

\def\bysame{\leavevmode ---------\thinspace}
\makeatletter\if@francais\providecommand{\og}{<<~}\providecommand{\fg}{~>>}
\else\gdef\og{``}\gdef\fg{''}\fi\makeatother
\def\cdrandname{\&}
\providecommand\cdrnumero{no.~}
\providecommand{\cdredsname}{eds.}
\providecommand{\cdredname}{ed.}
\providecommand{\cdrchapname}{chap.}
\providecommand{\cdrmastersthesisname}{Memoir}
\providecommand{\cdrphdthesisname}{PhD Thesis}
\begin{thebibliography}{10}

\bibitem{Ada}
{\scshape D.~R. Adams}, {\og A note on Riesz potentials\fg}, \emph{Duke Math.
  J.} \textbf{4} (1975), p.~765-778.

\bibitem{AH}
{\scshape D.~R. Adams {\normalfont \cdrandname}~L.~I. Hedberg}, \emph{Function
  Spaces and Potential Theory}, Grundlehren der math. Wissenschaften
  \textbf{314}, Springer, Berlin-Heidelberg-New York, 1996.

\bibitem{AMP}
{\scshape K.~Adimurthi, T.~Mengesha {\normalfont \cdrandname}~N.~C. Phuc}, {\og
  Gradient weighted norm inequalities for linear elliptic equations with
  discontinuous coefficients\fg}, \emph{Appl. Math. Optim.} \textbf{83} (2021),
  p.~327-371.

\bibitem{AP}
{\scshape K.~Adimurthi {\normalfont \cdrandname}~N.~C. Phuc}, {\og Global
  Lorentz and Lorentz-Morrey estimates below the natural exponent for
  quasilinear equations\fg}, \emph{Calc. Var. PDE} \textbf{54} (2015),
  p.~3107-3139.

\bibitem{CV}
{\scshape D.~T. Cao {\normalfont \cdrandname}~I.~E. Verbitsky}, {\og Nonlinear
  elliptic equations and intrinsic potentials of Wolff type\fg}, \emph{J.
  Funct. Anal.} \textbf{272} (2017), p.~112-165.

\bibitem{Giu}
{\scshape E.~Giusti}, \emph{Direct Methods in the Calculus of Variations},
  World Scientific, Singapore, 2003.

\bibitem{HL}
{\scshape Q.~Han {\normalfont \cdrandname}~F.~Lin}, \emph{Elliptic Partial
  Differential Equations, Second ed.}, Amer. Math. Soc., Providence, RI, 2011.

\bibitem{HeWo}
{\scshape L.~I. Hedberg {\normalfont \cdrandname}~T.~H. Wolff}, {\og Thin sets
  in nonlinear potential theory\fg}, \emph{Ann. Inst. Fourier (Grenoble)}
  \textbf{33} (1983), p.~161-187.

\bibitem{HKM}
{\scshape J.~Heinonen, T.~Kilpel{\"a}inen {\normalfont \cdrandname}~O.~Martio},
  \emph{Nonlinear Potential Theory of Degenerate Elliptic Equations}, Oxford
  Univ. Press, Oxford, 1993.

\bibitem{JV}
{\scshape B.~J. Jaye {\normalfont \cdrandname}~I.~E. Verbitsky}, {\og Local and
  global behaviour of solutions to nonlinear equations with natural growth
  terms\fg}, \emph{Arch. Rational Mech. Anal.} \textbf{204} (2012), p.~627-681.

\bibitem{KKT}
{\scshape T.~Kilpel{\"a}inen, T.~Kuusi {\normalfont
  \cdrandname}~A.~Tuhola-Kujanp{\"a}{\"a}}, {\og Superharmonic functions are
  locally renormalized solutions\fg}, \emph{Ann. Inst. H. Poincar{\'e}, Anal.
  Non Lin{\'e}aire} \textbf{28} (2011), p.~775-795.

\bibitem{KM1}
{\scshape T.~Kilpel{\"a}inen {\normalfont \cdrandname}~J.~Mal{\'y}}, {\og
  Degenerate elliptic equations with measure data and nonlinear potentials\fg},
  \emph{Ann. Scuola Norm. Sup. Pisa, Cl. Sci. (4)} \textbf{19} (1992),
  p.~591-613.

\bibitem{KM2}
\bysame , {\og The Wiener test and potential estimates for quasilinear elliptic
  equations\fg}, \emph{Acta Math.} \textbf{172} (1994), p.~137-161.

\bibitem{KZh}
{\scshape T.~Kilpel{\"a}inen {\normalfont \cdrandname}~X.~Zhong}, {\og
  Removable sets for continuous solutions of quasilinear elliptic
  equations\fg}, \emph{Proc. Amer. Math. Soc.} \textbf{130} (2002),
  p.~1681-1688.

\bibitem{KuMi}
{\scshape T.~Kuusi {\normalfont \cdrandname}~G.~Mingione}, {\og Guide to
  nonlinear potential estimates\fg}, \emph{Bull. Math. Sci.} \textbf{4} (2014),
  p.~1-82.

\bibitem{DMOP}
{\scshape G.~D. Maso, F.~Murat, A.~Orsina {\normalfont
  \cdrandname}~A.~Prignet}, {\og Renormalized solutions of elliptic equations
  with general measure data\fg}, \emph{Ann. Scuola Norm. Sup. Pisa Cl. Sci.
  (4)} \textbf{28} (1999), p.~741-808.

\bibitem{Maz}
{\scshape V.~G. Maz'ya}, \emph{Sobolev Spaces, with Applications to Elliptic
  Partial Differential Equations, Second, revised and augmented edition},
  Grundlehren der math. Wissenschaften \textbf{342}, Springer, Heidelberg,
  2011.

\bibitem{Min1}
{\scshape G.~Mingione}, {\og The Calder{\'{o}}n-Zygmund theory for elliptic
  problems with measure data\fg}, \emph{Ann. Scuola Norm. Sup. Pisa, Cl. Sci.
  (5)} \textbf{6} (2007), p.~195-261.

\bibitem{Min2}
\bysame , {\og Gradient estimates below the duality exponent\fg}, \emph{Math.
  Ann.} \textbf{346} (2010), p.~571-627.

\bibitem{NP2}
{\scshape Q.-H. Nguyen {\normalfont \cdrandname}~N.~C. Phuc}, {\og Existence
  and regularity estimates for quasilinear equations with measure data: the
  case $1<p \leq \frac{ 3n-2}{2n-1}$\fg}, \emph{Anal. PDE} (to appear),
  available at https://arxiv.org/abs/2003.03725.

\bibitem{NP1}
\bysame , {\og Good-$\lambda$ and Muckenhoupt-Wheeden type bounds in
  quasilinear measure datum problems, with applications\fg}, \emph{Math. Ann.}
  \textbf{374} (2019), p.~67-98.

\bibitem{Ph}
{\scshape N.~C. Phuc}, {\og Morrey global bounds and quasilinear Riccati type
  equations below the natural exponent\fg}, \emph{J. Math. Pures Appl.}
  \textbf{102} (2014), p.~99-123.

\bibitem{PV1}
{\scshape N.~C. Phuc {\normalfont \cdrandname}~I.~E. Verbitsky}, {\og
  Quasilinear and Hessian equations of Lane-Emden type\fg}, \emph{Ann. Math.}
  \textbf{168} (2008), p.~859-914.

\bibitem{PV2}
\bysame , {\og Singular quasilinear and Hessian equations and inequalities\fg},
  \emph{J. Funct. Anal.} \textbf{256} (2009), p.~1875-1906.

\bibitem{Tru}
{\scshape N.~S. Trudinger}, {\og On Harnack type inequalities and their
  applications to quasilinear elliptic equations\fg}, \emph{Comm. Pure Appl.
  Math.} \textbf{20} (1967), p.~721-747.

\bibitem{TW}
{\scshape N.~S. Trudinger {\normalfont \cdrandname}~X.~J. Wang}, {\og On the
  weak continuity of elliptic operators and applications to potential
  theory\fg}, \emph{Amer. J. Math.} \textbf{124} (2002), p.~369-410.

\bibitem{V2}
{\scshape I.~E. Verbitsky}, {\og Bilateral estimates of solutions to
  quasilinear elliptic equations with sub-natural growth terms\fg}, \emph{Adv.
  Calc. Var.} (to appear), available at https://doi.org/10.1515/acv-2021-0004.

\bibitem{V1}
\bysame , {\og Quasilinear elliptic equations with sub-natural growth terms and
  nonlinear potential theory\fg}, \emph{Atti Accad. Naz. Lincei, Rend. Lincei,
  Mat. Appl.} \textbf{30} (2019), p.~733-758.

\end{thebibliography}
\end{document}